\documentclass[a4paper]{article}
\usepackage{mflogo,texnames}
\usepackage{amsmath,amssymb,amsthm}
\usepackage{color}
\usepackage[colorlinks=true]{hyperref}

\numberwithin{equation}{section}

\newtheorem{theorem}{Theorem}[section]
\newtheorem{lemma}{Lemma}[section]
\newtheorem{proposition}{Proposition}[section]

\newtheorem{corollary}{Corollary}[section]
\newtheorem{remark}{Remark}[section]

\hypersetup{linkcolor=blue,urlcolor=red,citecolor=red}
\allowdisplaybreaks[4] \numberwithin{equation}{section}

\title{\textbf{Interpolation inequality at one time point for parabolic equations with time-independent coefficients and applications\thanks{The first author acknowledges the financial support by 
the National Natural Science Foundation of China under grants 11601377. The second author acknowledges the financial support by the National Natural Science Foundation of China under grants 11501424, and by Ministerio de Ciencia e Innovaci\'on grant MTM2014-53145-P, Spain.
}}}

\author{Huaiqiang Yu\thanks{School of Mathematics, Tianjin University, Tianjin 300354, China. \emph{Email: huaiqiangyu@tju.edu.cn}}\qquad
Can Zhang\thanks{School of Mathematics and Statistics, Wuhan University, Wuhan, 430072, China,
and Department of Mathematics, University of the Basque Country (UPV/EHU),  Bilbao, 48080, Spain.
\emph{Email: zhangcansx@163.com}}\quad
}

\begin{document}
\maketitle
\begin{abstract}
In this paper, we study the H\"older-type interpolation inequality and observability inequality from measurable sets in time for parabolic equations either with $L^p$ unbounded potentials or with electric potentials. The parabolic equations under consideration evolve in  bounded $C^{1,1}$ domains of 
$\mathbb R^N (N\geq3)$ with homogeneous Neumann boundary conditions. 
The approach for   the  interpolation inequality is based on a modified 
reduction method and some stability estimates for the corresponding elliptic operator.
\end{abstract}
\vskip 8pt
    \noindent

    
    \medskip
    
\vskip 10pt

\tableofcontents

\section{Introduction and main results}\label{sec_intro}

  Let $\Omega\subset\mathbb{R}^N$ ($N\geq 3$) be a bounded domain  with a
     $C^{1,1}$ boundary $\partial \Omega$ and such that $0\in \Omega$. 
For any $T>0$, consider the   following parabolic equation with time-independent coefficients and homogeneous conormal Neumann boundary condition 
\begin{equation}\label{yu-6-24-1}
\begin{cases}
	u_t-\mbox{div}(A(x)\nabla u)+b(x)u=0&\mbox{in}\;\;\Omega\times(0,T),\\
	A\nabla u\cdot\nu=0\;\;&\mbox{on}\;\;\partial\Omega\times(0,T),\\
	u(\cdot,0)=u_0\in L^2(\Omega),
\end{cases}
\end{equation}
where $\nu$ is the exterior unit normal vector on $\partial\Omega$,  the  symmetric matrix-valued function $A: \overline\Omega\rightarrow \mathbb{R}^{N\times N}$ is  Lipschitz continuous and satisfies the uniform ellipticity condition, i.e.,
	  there is a constant $\Lambda_1>1$ such that
\begin{equation}\label{yu-11-28-2}
\begin{cases}
	|a_{ij}(x)-a_{ij}(y)|\leq \Lambda_1|x-y|\;\;\mbox{for all}\;\;x,y\in\Omega\;\;\mbox{and each}\;\;i,j=1,\cdots, N,\\
    \Lambda_1^{-1}|\xi|^2\leq A(x)\xi\cdot\xi \leq \Lambda_1|\xi|^2\;\;\mbox{for a.e.}\;x\in\Omega\;\;\mbox{and all}\;\;\xi\in\mathbb{R}^N,
\end{cases}
\end{equation}
the unbounded potential $b(\cdot)$ verifies one of the following two assumptions: 
\begin{equation}\label{yu-6-24-1-1-b}
\begin{cases}
(i)~~\|b(\cdot)\|
_{L^{N+\delta}(\Omega)}\leq\Lambda_2\;\;\text{for some} \;\delta>0; \\
(ii)~~|b(x)|\leq \dfrac{\Lambda_2}{|x|}\;\;\;\;\mbox{for a.e.}\; x\in\Omega
\end{cases}
\end{equation}
with  $\Lambda_2>0$.

The first goal of the present paper is to establish a H\"older-type interpolation inequality at one time point for all solutions $u$ to \eqref{yu-6-24-1}.  Roughly speaking, for any $t>0$, there exist  constants $C>0$ and $\theta\in (0,1)$ such that
$$
\|u(\cdot,t)\|_{L^2(\Omega)}\leq C\|u(\cdot,t)\|_{L^2(B_R(x_0)\cap\Omega)}^{\theta}\|u_0\|_{L^2(\Omega)}^{1-\theta}\quad\text{for all}\;\;u_0\in L^2(\Omega).
$$
 Such a kind of interpolation inequality have been established 
for solutions of parabolic equations  either in convex bounded domains or in bounded $C^2$-smooth domains but with homogeneous Dirichlet boundary conditions; See for instance \cite{Bardos-Phung, Phung-2017,  Phung-Wang-2010, Phung-Wang-2013, Phung-Wang-Zhang,  Zhang-2017}.   In these papers, the approach for 
the desired interpolation inequality is mainly based on the parabolic-type Almgren frequency function method, which is essentially adapted from \cite{Escauriaza-Fernandez-Vessella-2006,Poon}.  

The second goal of this paper is to deduce an observability inequality from measurable sets in time.
This can be immediately obtained from the above-mentioned interpolation inequality combined with
the telescoping series method developed in \cite{Apraiz-Escauriaza-Wang-Zhang, Phung-Wang-2013}.

More precisely,  the main results of this paper can be stated as follows.

\begin{theorem}\label{jiudu4}
Let $T>0$ and $\omega\subset \Omega$ be a non-empty open subset. Then  there are constants $C=C(\Lambda_1,\Lambda_2,N,\delta,\Omega,\omega)>0$ and $\sigma=\sigma(\Lambda_1,\Lambda_2,N,\delta,\Omega,\omega)\in(0,1)$ such that 
	for any solution $u$ of (\ref{yu-6-24-1}) with the initial value $u_0\in L^2(\Omega)$,
	\begin{equation}\label{yu-7-12-1}
	\|u(\cdot,t_0)\|_{L^2(\Omega)}\leq Ce^{\frac{C(T^2+1)}{t_0}}\|u(\cdot,t_0)\|_{L^2(\omega)}^\sigma\|u_0\|_{L^2(\Omega)}^{1-\sigma}\;\;\mbox{for all}\;\;t_0\in (0,T).
\end{equation}
\end{theorem}
\begin{remark}
	In \cite{Phung-Wang-2013}, the authors have obtained the  global interpolation inequality (\ref{yu-7-12-1}) for the heat equation with zero Dirichlet boundary condition and $L^\infty(0,T;L^p(\Omega))$ potential under the assumption $p>N$. 
	This is coincident with the assumption (i) in (\ref{yu-6-24-1-1-b}). However, in view of the electric potential $O(|x|^{-1})$, one could see that the assumption $p>N$ is not optimal (see also \cite{Zhang-2017}). 
\end{remark}

%Based on Theorem \ref{jiudu4}, we have the following  observability estimate 
%for the equation (\ref{yu-6-24-1}):

\begin{theorem}\label{yu-main-1}
Assume $\omega\subset\Omega$ is a non-empty open subset. Let $T>0$  and $E\subset [0,T]$ be a subset of positive measure. Then there is a constant $C=C(\Lambda_1,\Lambda_2,N,\delta,\Omega,\omega,T,E)>0$ such that for any solution $u$ of 
(\ref{yu-6-24-1}) with the initial value $u_0\in L^2(\Omega)$,
\begin{equation}\label{yu-6-24-2}
	\|u(\cdot,T)\|_{L^2(\Omega)}\leq C\int_E\|u(\cdot,t)\|_{L^2(\omega)}dt.
\end{equation}
	In particular, when $E=[0,T]$, the constant $C$ in the above inequality can be taken the form 
	$$C(\Lambda_1,\Lambda_2,N,\delta,\Omega,\omega)e^{\frac{C(\Lambda_1,\Lambda_2,N,\delta,\Omega,\omega)(T^2+1)}{T}}.$$
\end{theorem}
\bigskip
It follows from the classical Hilbert uniqueness method (HUM) that (see, e.g., \cite{Apraiz-Escauriaza-Wang-Zhang})
\begin{corollary}
Let $T>0$. Assume $\omega\subset\Omega$ is a nonempty open subset and $E\subset [0,T]$ is a subset of positive measure. Then, for any $u_0\in L^2(\Omega)$, there is a control $f\in L^\infty(0,T;L^2(\Omega))$, with 
$$\|f\|_{L^\infty(0,T;L^2(\Omega))}\leq C\|u_0\|_{L^2(\Omega)} \quad\text{for  the same constant}\;C\;\mbox{appeared in (\ref{yu-6-24-2})},$$
such that the solution of 
\begin{equation*}
\begin{cases}
	u_t-\mbox{div}(A(x)\nabla u)+b(x)u=\chi_{E\times \omega}f    &\mbox{in}\;\;\Omega\times(0,T),\\
	A\nabla u\cdot\nu=0&\mbox{on}\;\;\partial\Omega\times(0,T),\\
	u(\cdot,0)=u_0\in L^2(\Omega)
\end{cases}
\end{equation*}
satisfies $u(x,T)=0$ for a.e. $x\in\Omega$. 
\end{corollary}

\par

%There are several starting motivations for the study of the present paper.  The first one is as follows. 
 The interpolation inequality (\ref{yu-7-12-1})  at one time point in Theorem \ref{jiudu4} is a quantitative form 
  of  strong unique continuation for the equation (\ref{yu-6-24-1}). The study of unique continuation
  property for parabolic equation has a long history. 
  For the works in this topic, one can see 
  \cite{Adolfsson-Escauriaza, Alessandrini-Vessella, Canuto-Rosset-Vessella,  Chen-1996,  Escauriaza-2000,  
  Escauriaza-Fernandez-Vessella-2006, Escauriaza-Montaner-Zhang, Escauriaza-Vessella-2003,  Ito-Yamabe-1958,   
 Lin-1990,   Lin-1991, Poon,  Vessella-2009, Yamabe-1959} and references therein. Among these papers, it is worth mentioning particularly \cite{Lin-1990} and 
  \cite{Canuto-Rosset-Vessella}.   In the paper \cite{Lin-1990}, F. H. Lin showed the strong unique continuation property for
  the equation (\ref{yu-6-24-1}) when the potential $b(\cdot)\in L_{loc}^{(N+1)/2}(\Omega)$. Although it is a qualitative
  form of unique continuation,  F. H. Lin constructed  an important and smart strategy that deduces a strong unique continuation of 
  parabolic equations with time-independent coefficients to the elliptic counterparts. 
    Later, by following and quantifying this strategy, 
  B. Canuto, E. Rosset and S. Vessella proved  in \cite{Canuto-Rosset-Vessella} the 
  local quantitative unique continuation for time-independent parabolic equations but without potentials (i.e.,  $b(\cdot)=0$).
   It seems to us that the results in \cite{Canuto-Rosset-Vessella} are not enough to derive
   the interpolation inequality in Theorem \ref{jiudu4}; See more discussions in Remark \ref{d1} below. Further, 
  the presence of potential term will lead to some difficulties if one follows the same argument used in \cite{Canuto-Rosset-Vessella}. These difficulties force us to slightly improve the strategy used by B. Canuto, E. Rosset and S. Vessella  (see Section \ref{dujin1} below).

When the boundary condition  in \eqref{yu-6-24-1} is homogeneous Dirichlet-type,  through using the frequency function method, the global interpolation inequality in Theorem \ref{jiudu4}  has been studied in \cite{Bardos-Phung, Phung-2017, Phung-Wang-2010, Phung-Wang-2013, Phung-Wang-Zhang, Zhang-2017}. However,
to the best of our knowledge, this approach seems to be not applicable for the case of homogeneous Neumann boundary condition (at least we do not know). This forces us to find a new method to obtain the corresponding interpolation inequality.

      In order to overcome  these difficulties mentioned above, in this paper we shall adopt and slightly modify the reduction method, as well as Carleman estimates of elliptic operators. Roughly speaking, the reduction method \cite{Lin-1990}  is to 
       reduce 
a parabolic equation into an elliptic equation by using the Fourier transformation and adding one more spatial variable. However, 
      because of the appearance of potential term, we shall adopt a sinh-type weighted  Fourier transformation, which is slightly  different to 
      the strategy used in \cite{Lin-1990,Canuto-Rosset-Vessella}. 
      Moreover, 
      for the proof of stability estimate (see Lemma \ref{yu-proposition-7-1-1} below), 
      the authors of \cite{Canuto-Rosset-Vessella} reduced the elliptic equation to a hyperbolic equation and used harmonic measure.  This strategy, in our opinion, cannot be applied when the potential is nonzero.   Instead, in this paper we shall use suitable Carleman estimates
      to deduce the corresponding stability estimate. 
      Note that the reduction method is based on a representation formula for solutions of parabolic equations in terms 
of eigenfunctions of the corresponding elliptic operators, and therefore cannot be applied to 
general parabolic equations with time-dependent coefficients.  

We emphasis   that in the case  of heat equation with homogeneous Dirichlet boundary conditions, the authors in \cite{Apraiz-Escauriaza-Wang-Zhang} first observed  that the observability estimate at one time point is in fact equivalent to 
  a type of spectral inequality in \cite{Lebeau-Robbiano-1998} (see also \cite{Phung-2017}).  This type of spectral inequality, roughly speaking,  is an observability inequality from a partial region on the finite sum of eigenfunctions of the principal elliptic operator. 
  For related works, we refer the reader to \cite{Chaves-Silva-Lebeau-2016, Rousseau-Robbiano-2012, Lebeau-Zuazua, LU-2013, Miller}  and references therein. 
Therefore, if one could establish a type of spectral inequality as in \cite{Lebeau-Robbiano-1998} (see also \cite{Phung-2017}), the global interpolation inequality can also be deduced by the technique utilized in 
  \cite{Apraiz-Escauriaza-Wang-Zhang}. We refer \cite{LU-2013} for the spectral inequality of elliptic equation with Neumann boundary condition and without any potential term. 
  
Meanwhile, we also refer \cite{Escauriaza-Fernandez-Vessella-2006,Vessella-2009} for quantitative estimates of unique continuation
of parabolic equations with time dependent coefficients, in which some parabolic-type Carleman estimates were established. 
We believe that the Carleman method developed in \cite{Vessella-2009} (or \cite{Escauriaza-Fernandez-Vessella-2006}) may provide a possible approach for proving the corresponding interpolation inequality.
However, this issue escapes the study of the present paper and is deserved to be  investigated  in the continued work.

Last but not least, we would like to stress that 
the observability estimate from measurable sets in the time variable established in  Theorem \ref{yu-main-1} has several 
 applications in control theory. 
 In particular, it
  implies  bang-bang properties of minimal norm and minimal time optimal control problems (see for instance \cite{Phung-Wang-2013,wang-zhang1}).

 \smallskip

        The structure of this paper is organized as follows.   
       In Section \ref{mainproof}, we first present two quantitative estimates of unique continuation needed for proof of the main results,  and then we prove Theorems \ref{jiudu4} and \ref{yu-main-1}, respectively.  In Section \ref{dujin1}, we are devoted to the proofs of the above-mentioned two quantitative estimates of unique continuation. In Appendix, the proofs of some results used in Section \ref{dujin1} are given.

\paragraph{Notation.}Throughout the paper,
$\triangle_R(x_0)$ stands for a ball in $\mathbb R^N$ with the center $x_0$ and of radius $R>0$,
$B_R(x_0,0)$ stands for a ball in $\mathbb R^{N+1}$ with the center $(x_0,0)$ and of radius $R>0$. 
Denote by $\partial\triangle_R(x_0)$ the boundary of $\triangle_R(x_0)$,  by $\rho_0=\sup\{|x-y|:x,y\in\Omega\}$ and $\Omega_\rho=\{x\in\Omega\,:\,d(x,\partial\Omega)\geq\rho$\} with $\rho\in(0,\min\{1,\rho_0\})$.
Write $\bar z$ for the complex conjugate of a complex number $z\in\mathbb C$.
The letter $C$ denotes a generic positive
constant that depends on the a priori data but not on the solution and may vary from line to line.
Moreover, we shall denote by $C(\cdot)$ a positive constant if we need to emphasize the dependence on some parameters in the brackets.

\section{Proofs of main results}\label{mainproof}

\subsection{Unique continuation estimates}
In order to present the proof of Theorem \ref{jiudu4},
we first state two results concerning quantitative estimates of unique continuation: The first one is local, and the second one is global.  Their proofs are postponed to give in Section \ref{dujin1}.

\begin{proposition}\label{yu-theorem-7-5-1}
Let $T>0$. Suppose $\rho\in(0,\min\{1,\rho_0\})$ such that $\Omega_\rho\neq \emptyset$. Then there exist $R\in(0,\rho)$ and $\kappa\in(0,1/2)$ such that for any $r\in (0, \kappa R)$, any $t_0\in(0,T/2)$ and any $x_0\in\Omega_\rho$, we have
\begin{equation*}\label{yu-7-5-12}
	\|u(\cdot,t_0)\|_{L^2(\triangle_{2r}(x_0))}\leq Ce^{\frac{C(T^2+1)}{t_0}}
	\|u(\cdot,t_0)\|_{L^2(\triangle_{r}(x_0))}^{\sigma}\left(\sup_{s\in[0,T]}\|u(\cdot,s)\|_{H^1(\triangle_R(x_0))}\right)^{1-\sigma}
\end{equation*}
with some constants $C=C(\Lambda_1,\Lambda_2, \Lambda_3, N,\delta,r,R)>0$ and $\sigma=\sigma(\Lambda_1,\Lambda_2,\Lambda_3,N,\delta,r)\in(0,1)$, where $u\in C([0,2T];H_{\text{loc}}^1(\Omega))$ satisfies 
\begin{equation}\label{heat capacity}
l(x)\partial_tu-\mbox{div}(A(x)\nabla u)+b(x)u=0\;\;\;\mbox{in}\;\;\Omega\times(0,2T).
\end{equation}	
Here $A$ and $b$ are the same as in \eqref{yu-6-24-1}, and $l:\Omega\rightarrow\mathbb R^+$ verifies  
\begin{equation}\label{yu-7-29-3}
\Lambda_3^{-1}\leq l(x)\leq \Lambda_3,\;\;
	|l(x)-l(y)|\leq \Lambda_3|x-y|
\quad \;\;\mbox{for a.e.}\;\;x,y\in\Omega
\end{equation}
with a constant $\Lambda_3>1$. 
\end{proposition} 

\begin{proposition}\label{yu-theorem-7-10-6}
Let $T>0$ and $\omega\subset\Omega$ be a non-empty open subset. 
Then there  are constants $C=C(\Lambda_1,\Lambda_2,N,\delta,\Omega,\omega)>0$ and  $\sigma=\sigma(\Lambda_1,\Lambda_2,N,\delta,\Omega,\omega)\in(0,1)$ such that for any solution  $u\in C([0,T];H^1(\Omega))$ of \eqref{yu-6-24-1}, we have
\begin{equation*}\label{yu-7-10-2}
	\|u(\cdot,t_0)\|_{L^2(\Omega)}\leq C
	e^{\frac{C(T^2+1)}{t_0}}\|u(\cdot,t_0)\|_{L^2(\omega)}^\sigma 
	\left(\sup_{s\in[0,T]}\|u(\cdot,s)\|_{H^1(\Omega)}\right)^{1-\sigma}\quad \text{for all}\;\; t_0\in(0,T/2).
\end{equation*}
\end{proposition}

\begin{remark}\label{d1}
%(i) Inspired by the approachs in \cite{vessella,lebeau,Luis-convex}, we adapt them to our different situation.
%\vskip 5pt
	The local interpolation inequality established in Proposition \ref{yu-theorem-7-5-1}
	is slightly different from the two spheres and one cylinder inequality established  in \cite[Theorem 3.1.1']{Canuto-Rosset-Vessella}.
Actually, in \cite{Canuto-Rosset-Vessella}, the bound for the parameter $r$ depends 
	on the instant $t_0$. This, however, will lead some difficulties when one applies it to prove the global interpolation  and 
	 observability inequalities. 
	\end{remark}

\begin{remark}
Equations of type \eqref{heat capacity} appear when one transforms the parabolic operator via a linear 
mapping from $\mathbb R^N$ into $\mathbb R^N$.
It is also worth mentioning that parabolic equations of form \eqref{heat capacity} with positive coefficients in front of the time derivative  are much more nature from the physical point of view.  
They model the heat diffusion of the temperature in a non-isotropic and non-homogeneous material.
In fact, there are two relevant physical quantities in heat diffusion processes: the conductivity coefficients and the specific heat capacity. The latter appears in the equation in front of the time derivative.
\end{remark}

\subsection{Proof of Theorem \ref{jiudu4}}

We first recall the following well-known Hardy inequality (see e.g. \cite{Davies}) and Sobolev interpolation theorem (see e.g. \cite[Theorem 5.8]{Adams-Fournier}), which will be used frequently  in our argument below.
\begin{lemma}\label{hardy}
Let $\Omega$ be a bounded Lipschitz domain in $\mathbb R^N$ ($N\geq 3$). Then, it holds that
\vskip 5pt 
 (i) (Hardy's inequality) 
$$
\int_{\Omega}|x|^{-2}|f|^2dx\leq 
\frac{4}{(N-2)^2}\int_{\Omega}|\nabla f|^2dx\quad \mbox{for any}\;\;
f\in H_0^1(\Omega).
$$
\vskip 5pt
	(ii) (Sobolev's interpolation theorem) For each $p\in [2,\frac{2N}{N-2}]$, there is a constant $\Gamma_1(\Omega, N, p)>0$ such that 
$$
	\|f\|_{L^{p}(\Omega)}\leq \Gamma_1(\Omega, N, p)\|f\|_{H^1(\Omega)}^\theta\|f\|_{L^2(\Omega)}^{1-\theta}\;\;\mbox{for any}\;\;
f\in H^1(\Omega).
$$
	where $\theta=N(\frac{1}{2}-\frac{1}{p})$.
\end{lemma}
As a simple consequence of  the above Sobolev interpolation theorem, we have 
\begin{corollary}\label{yu-corollary-1}
	Let $\Omega$ be a bounded Lipschitz domain in $\mathbb R^N$ ($N\geq 3$). For each case of 
	(\ref{yu-6-24-1-1-b}) and for every $\epsilon\in\left(0,\frac{1}{2}\right]$,  it holds that 
\begin{equation*}\label{yu-9-26-1}
	b(\cdot)\in L^{\frac{N}{2}+\epsilon}(\Omega).
\end{equation*}
Further, for each $\eta>0$ there is a constant $\Gamma_2(\Omega, N, \eta)>0$ such that,   for any $h(\cdot)\in L^{\frac{N}{2}+\eta}(\Omega)$ 
	 and $f(\cdot)\in H^1(\Omega)$,
\begin{eqnarray}\label{yu-9-26-2}
	\int_{\Omega}|h||f|^2dx
	%\leq\|h\|_{L^{\frac{N}{2}+\epsilon}(\Omega)}\|f\|_{L^{\frac{2(N+2\epsilon)}{N+2(\epsilon-1)}}(\Omega)}^2
	\leq \Gamma_2(\Omega, N,\eta)\|h\|_{L^{\frac{N}{2}+\eta}(\Omega)}\|f\|_{L^2(\Omega)}^{\frac{4\eta}{N+2\eta}}\|f\|_{H^1(\Omega)}^{\frac{2N}{N+2\eta}}.
\end{eqnarray}
\end{corollary}

\medskip

\begin{proof}[\textbf{Proof of Theorem \ref{jiudu4}}]	The  proof  is divided into two steps. 
\vskip 5pt
	\textbf{Step 1. Energy estimates.} 
	In this step, we shall prove the following two claims:
\begin{itemize}
  \item If $u(\cdot,0)\in L^2(\Omega)$, then for each $t\in[0,6T]$ we have
\begin{equation}\label{yu-7-12-2-b}
	\|u(\cdot,t)\|_{L^2(\Omega)}\leq e^{Ct}\|u(\cdot,0)\|_{L^2(\Omega)}
\end{equation}
	and
\begin{equation}\label{yu-7-12-2}
	\|u(\cdot,t)\|_{H^1(\Omega)}\leq \frac{Ce^{Ct}}{\sqrt{t}}\|u(\cdot,0)\|_{L^2(\Omega)}.
	\end{equation}
  \item If $u(\cdot,0)\in H^1(\Omega)$, then we have
 \begin{equation}\label{yu-7-12-3}
 	\|u(\cdot,t)\|_{H^1(\Omega)}\leq e^{Ct}\|u(\cdot,0)\|_{H^1(\Omega)}\;\;\mbox{for each}\;\;
	t\in[0,6T].
 \end{equation}
 \end{itemize}
Indeed, multiplying the first equation 
	 in (\ref{yu-6-24-1}) by $u$ and then integrating by parts over $\Omega$, we get 
\begin{equation}\label{yu-7-15-1}
	\frac{1}{2}\frac{d}{dt}\int_{\Omega}|u|^2dx
	+\int_{\Omega}\nabla u\cdot(A\nabla u)dx=-\int_{\Omega}b|u|^2dx.
\end{equation}
Note that, from (\ref{yu-9-26-2}) (by letting $\eta=\frac{1}{2}$ there) and (\ref{yu-11-28-2}), we have 
\begin{eqnarray*}\label{yu-7-24-1}
	-\int_\Omega b|u|^2dx&\leq& C\|b\|_{L^{\frac{N+1}{2}}(\Omega)}\|u\|^{\frac{2}{N+1}}_{L^2(\Omega)}
	\|u\|^{\frac{2N}{N+1}}_{H^1(\Omega)}\leq C\Lambda_2\|u\|_{L^2(\Omega)}^{\frac{2}{N+1}}\|u\|_{H^1(\Omega)}^{\frac{2N}{N+1}}\nonumber\\
	&\leq& C\epsilon^{-N}\Lambda_2^{N+1}\|u\|_{L^2(\Omega)}^2+\epsilon\|u\|_{H^1(\Omega)}^2
	=(C\epsilon^{-N}\Lambda_2^{N+1}+\epsilon)\|u\|^2_{L^2(\Omega)}+\epsilon\|\nabla u\|^2_{L^2(\Omega)}\nonumber\\
	&\leq&(C\epsilon^{-N}\Lambda_2^{N+1}+\epsilon)\|u\|^2_{L^2(\Omega)}+\epsilon\Lambda_1\int_\Omega\nabla u\cdot(A\nabla u)dx.
\end{eqnarray*}
	Taking $\epsilon=\frac{1}{2\Lambda_1}$ in the above inequality, we obtain that 

\begin{equation*}\label{yu-7-24-2}
	-\int_{\Omega}b|u|^2dx\leq 
	 \left[C\Lambda_1^N\Lambda_2^{N+1}+\frac{1}{2\Lambda_1}\right]
	\int_{\Omega}|u|^2dx+\frac{1}{2}\int_{\Omega}\nabla u\cdot(A\nabla u)dx.
\end{equation*}
	This, along with (\ref{yu-7-15-1}), yields
\begin{equation}\label{yu-7-15-2}
	\frac{1}{2}\frac{d}{dt}\int_{\Omega}|u|^2dx
	+\frac{1}{2}\int_{\Omega}\nabla u\cdot(A\nabla u)dx\leq  \left[C\Lambda_1^N\Lambda_2^{N+1}
	+\frac{1}{2\Lambda_1}\right]\int_{\Omega}|u|^2dx.
\end{equation}
	Then 
\begin{equation}\label{yu-7-15-3}
	\frac{d}{dt}\left(e^{-\left[C\Lambda_1^N\Lambda_2^{N+1}
	+\Lambda_1^{-1}\right]t}\int_{\Omega}|u|^2dx)\right)\leq 0.
\end{equation}
	This gives (\ref{yu-7-12-2-b}). Moreover, by (\ref{yu-7-15-2}) and (\ref{yu-7-15-3}), we obtain 
\begin{eqnarray}\label{yu-7-15-4}
	&\;&\int_0^t\int_{\Omega}\nabla u\cdot(A\nabla u)dxds\nonumber\\
	&\leq&
	\left[t\left(C\Lambda_1^N\Lambda_2^{N+1}+\Lambda_1^{-1}\right)e^{(C\Lambda_1^N\Lambda_2^{N+1}+\Lambda_1^{-1})t}+1
	\right]
	\|u(\cdot, 0)\|^2_{L^2(\Omega)}. 
\end{eqnarray}

	Next, we show (\ref{yu-7-12-3}). Here we divide our proof into two cases based on the assumptions in (\ref{yu-6-24-1-1-b}). 
\vskip 5pt
	\emph{Case I. $|b(x)|\leq \frac{\Lambda_2}{|x|}$ a.e. $x\in\Omega$.}
	In this case, we take $r_0\in(0,d(0,\partial\Omega))$ and $\eta\in C^\infty(\mathbb{R}^N;[0,1])$ such that 
\begin{equation*}\label{yu-7-24-4}
\begin{cases}
	\overline{\triangle_{r_0}(0)}\subset\Omega,\\
	\eta=1&\mbox{in}\;\;\triangle_{\frac{r_0}{2}}(0),\\
	\eta=0&\mbox{in}\;\;\mathbb{R}^N\backslash\triangle_{r_0}(0)\\
	|\nabla\eta |\leq \frac{C}{r_0}&\mbox{in}\;\;\mathbb{R}^N. 
\end{cases}
\end{equation*}
         Multiplying the first equation of (\ref{yu-6-24-1}) by $-\mbox{div}(A\nabla u)\eta^2$
	 and integrating by parts over 
	$\Omega$, by Lemma \ref{hardy} we have
\begin{eqnarray}\label{yu-7-24-5}
	&\;&\frac{1}{2}\frac{d}{dt}\int_{\Omega}\nabla u\cdot (A\nabla u)\eta^2dx
	+\int_{\Omega}|\mbox{div}(A\nabla u)|^2\eta^2dx\nonumber\\
	&\leq&2\int_{\Omega}|\mbox{div}(A\nabla u)||(A\nabla u)\cdot \nabla\eta|\eta dx
	+2\int_{\Omega}|b||u||(A\nabla u)\cdot\nabla\eta|\eta dx
	+\int_{\Omega}|b||u|\eta^2|\mbox{div}(A\nabla u)|dx\nonumber\\
	&\leq&\frac{1}{2}\int_{\Omega}|\mbox{div}(A\nabla u)|^2\eta^2dx+5\int_{\Omega}|(A\nabla u)\cdot\nabla \eta|^2dx
	+2\int_{\Omega}|b|^2|u|^2\eta^2dx\nonumber\\
	&\leq&\frac{1}{2}\int_{\Omega}|\mbox{div}(A\nabla u)|^2\eta^2dx
	+\frac{5C\Lambda_1}{r_0^2}\int_{\Omega}\nabla u\cdot(A\nabla u)dx
	+\frac{16\Lambda_1\Lambda_2^2}{(N-2)^2}\int_{\Omega}\nabla u\cdot(A\nabla u)\eta^2dx\nonumber\\
	&\;&+\frac{16C\Lambda_2^2}{(N-2)^2r_0^2}\int_{\Omega}|u|^2dx. 
\end{eqnarray}
	Further, multiplying  the first equation of (\ref{yu-6-24-1}) by $-\mbox{div}(A\nabla u)(1-\eta^2)$
	 and integrating by parts over 
	$\Omega$, we get	
\begin{eqnarray*}\label{yu-7-24-6}
	&\;&\frac{1}{2}\frac{d}{dt}\int_{\Omega}\nabla u\cdot (A\nabla u)(1-\eta^2)dx
	+\int_{\Omega}|\mbox{div}(A\nabla u)|^2(1-\eta^2)dx\nonumber\\
	&\leq&2\int_{\Omega}|\mbox{div}(A\nabla u)||(A\nabla u)\cdot \nabla\eta|\eta dx
	+2\int_{\Omega}|b||u||(A\nabla u)\cdot\nabla\eta|\eta dx
	+\int_{\Omega}|b||u|(1-\eta^2)|\mbox{div}(A\nabla u)dx|\nonumber\\
	&\leq&\frac{1}{4}\int_{\Omega}|\mbox{div} (A \nabla u)|^2\eta^2dx+5\int_{\Omega}|\nabla\eta\cdot(A\nabla u)|^2dx
	+\int_{\Omega}|b|^2|u|^2\eta^2dx\nonumber\\
	&\;&+\frac{\Lambda_2}{r_0}\int_{\Omega}|u|(1-\eta^2)|\mbox{div}(A\nabla u)|dx\nonumber\\
	&\leq&\frac{1}{4}\int_{\Omega}|\mbox{div} (A \nabla u)|^2\eta^2dx
	+\frac{5C\Lambda_1}{r_0^2}\int_{\Omega}\nabla u\cdot(A\nabla u)dx
	+\frac{8\Lambda_1\Lambda_2^2}{(N-2)^2}\int_{\Omega}\nabla u\cdot(A\nabla u)\eta^2dx\nonumber\\
	&\;&+\frac{8C\Lambda_2^2}{(N-2)^2r_0^2}\int_{\Omega}|u|^2dx
	+\frac{\Lambda_2^2}{3r_0^2}\int_{\Omega}|u|^2(1-\eta^2)dx
	+\frac{3}{4}\int_{\Omega}|\mbox{div}(A\nabla u)|^2(1-\eta^2)dx. 
\end{eqnarray*}
	This, together with (\ref{yu-7-24-5}), gives that 
\begin{eqnarray}\label{yu-7-24-7}
	&\;&\frac{1}{2}\frac{d}{dt}\int_{\Omega}\nabla u\cdot(A\nabla u)dx+\frac{1}{4}\int_{\Omega}|\mbox{div}(A\nabla u)|^2dx\nonumber\\
	&\leq&\left(\frac{10C\Lambda_1}{r_0^2}+\frac{24\Lambda_2^2\Lambda_1}{(N-2)^2}\right)
	\int_{\Omega}\nabla u\cdot (A\nabla u)dx
	+\left(\frac{24C}{(N-2)^2}+\frac{1}{3}\right)\frac{\Lambda_2^2}{r_0^2}
	\int_{\Omega}|u|^2dx.
\end{eqnarray}
	By (\ref{yu-7-15-2}) and (\ref{yu-7-24-7}), we get 
\begin{eqnarray*}\label{yu-7-24-8}
	\frac{1}{2}\frac{d}{dt}\left(\int_{\Omega}|u|^2dx
	+\int_{\Omega}\nabla u\cdot(A\nabla u)dx\right)
	\leq C(r_0) \left(\int_{\Omega}|u|^2dx
	+\int_{\Omega}\nabla u\cdot(A\nabla u)dx\right),
\end{eqnarray*}
	where
$$
	C(r_0):=C\Lambda_1^N\Lambda_2^{N+1}+\frac{1}{2\Lambda_1}+\left(\frac{24C}{(N-2)^2}+\frac{1}{3}\right)\frac{\Lambda_2^2}{r_0^2}
	+\frac{10C\Lambda_1}{r_0^2}
	+\frac{24\Lambda_2^2\Lambda_1}{(N-2)^2}.
$$
	Therefore, we have
\begin{equation}\label{yu-7-24-9}
	\frac{d}{dt}\left[e^{-2C(r_0)t}\int_\Omega\left(|u|^2+\nabla u\cdot(A\nabla u)\right)dx\right]\leq 0.
\end{equation}
	This gives  
\begin{eqnarray}\label{yu-7-15-8}
	\|u(\cdot,t)\|_{H^1(\Omega)}^2\leq  \Lambda_2^2e^{2C(r_0)t}\|u(\cdot,0)\|^2_{H^1(\Omega)}.
\end{eqnarray}
	 Hence (\ref{yu-7-12-3}) holds in this case. 
\vskip 5pt
	\emph{Case II. $b(\cdot)\in L^{N+\delta}(\Omega)$ and $\|b(\cdot)\|_{L^{N+\delta}(\Omega)}\leq \Lambda_2$.} 
	Multiplying  the first equation of (\ref{yu-6-24-1}) by $-\mbox{div}(A\nabla u)$ and integrating by parts over $\Omega$, by 
	Lemma \ref{hardy}, we get 	 
\begin{eqnarray*}\label{yu-9-28-1}
	&\;&\frac{1}{2}\frac{d}{dt}\int_{\Omega}\nabla u\cdot(A\nabla u)dx
	+\frac{1}{2}\int_{\Omega}|\mbox{div}(A\nabla u)|^2dx\nonumber\\
	&\leq&2\|b\|^2_{L^N(\Omega)}\|u\|^2_{L^{\frac{2N}{N-2}}(\Omega)}
	\leq 2\Gamma_1(\Omega, N, 2)\|b\|^2_{L^N(\Omega)}\|u\|^2_{H^1(\Omega)}\nonumber\\
	&\leq&C\Lambda_1\|b\|^2_{L^N(\Omega)}\int_{\Omega}(|u|^2+\nabla u\cdot (A\nabla u))dx.
\end{eqnarray*}
	 This, along with (\ref{yu-7-15-2}), gives that 	 
\begin{equation*}\label{yu-9-28-2}
	\frac{1}{2}\frac{d}{dt}\int_{\Omega}(|u|^2+\nabla u\cdot (A\nabla u))dx
	\leq \left(C\Lambda_2^2\Lambda_1+C\Lambda^N_1\Lambda_2^{N+1}+\frac{1}{2\Lambda_1}\right)
	\int_{\Omega}(|u|^2+\nabla u\cdot (A\nabla u))dx.
\end{equation*}
	This implies that 
\begin{equation}\label{yu-9-28-3}
	\frac{d}{dt}\left[e^{-(C\Lambda_2^2\Lambda_1+C\Lambda_1^N\Lambda_2^{N+1}+\Lambda_1^{-1})t}
	\int_{\Omega}(|u|^2+\nabla u\cdot (A\nabla u))dx\right]\leq 0.
\end{equation}
	 Similar to the proof of (\ref{yu-7-15-8}), we obtain (\ref{yu-7-12-3}) in this case. 
\par
Moreover, using (\ref{yu-7-24-9}) and (\ref{yu-9-28-3}) respectively in each case analyzed above, we obtain that there exists $C>0$ such that 
\begin{eqnarray*}\label{yu-7-15-9}
\int_0^t\int_{\Omega}(|u|^2+\nabla u\cdot (A\nabla u))dxds
	\geq \int_0^te^{-C(t-s)}ds\,  \|u(\cdot,t)\|_{H^1(\Omega)}^2	\geq \Lambda_1^{-1}te^{-Ct}\|u(\cdot,t)\|_{H^1(\Omega)}^2.
\end{eqnarray*}
	This, together with (\ref{yu-7-15-4}) and (\ref{yu-7-12-2-b}),
	 yields (\ref{yu-7-12-2}).

\vskip 5pt
	\textbf{Step 2. Completing the proof.} 
	We arbitrarily fixed $t_0\in(0,T)$ and consider the following equation
\begin{equation*}\label{yu-7-12-12}
\begin{cases}
	v_t-\mbox{div}(A(x)\nabla v)+bv=0&\mbox{in}\;\;\Omega\times(0,4T),\\
	A\nabla v\cdot\nu=0&\mbox{on}\;\;\partial\Omega\times(0,4T),\\
	v(\cdot,0)=u(\cdot,\frac{t_0}{2})&\mbox{in}\;\;\Omega.
\end{cases}
\end{equation*}
	It is obvious that $v(\cdot,t)=u(\cdot,t+\frac{t_0}{2})$ when $t\in[0,4T]$. Moreover, by (\ref{yu-7-12-2}) we have 
	$u(\cdot,\frac{t_0}{2})\in H^1(\Omega)$, 
	which means that $v\in C([0,4T];H^1(\Omega))$. 
	From Proposition \ref{yu-theorem-7-10-6}, it follows that there are $C>0$ and $\sigma\in(0,1)$ such that  
\begin{equation*}\label{yu-7-13-1}
	\left\|v\left(\cdot,\frac{t_0}{2}\right)\right\|_{L^2(\Omega)}\leq Ce^{\frac{C(T^2+1)}{t_0}}
	\left\|v\left(\cdot,\frac{t_0}{2}\right)\right\|^\sigma_{L^2(\omega)}\left(\sup_{s\in[0,T]}\|v(\cdot,s)\|_{H^1(\Omega)}\right)^{1-\sigma}. 
\end{equation*}
	This, along with (\ref{yu-7-12-3}), gives that 
\begin{equation*}\label{yu-7-13-2}
	\left\|v\left(\cdot,\frac{t_0}{2}\right)\right\|_{L^2(\Omega)}\leq Ce^{\frac{C(T^2+1)}{t_0}}
	\left\|v\left(\cdot,\frac{t_0}{2}\right)\right\|^\sigma_{L^2(\omega)}\|v(\cdot,0)\|^{1-\sigma}_{H^1(\Omega)}. \end{equation*}
Which is 
\begin{equation*}\label{yu-7-13-3}
	\left\|u\left(\cdot,t_0\right)\right\|_{L^2(\Omega)}\leq Ce^{\frac{C(T^2+1)}{t_0}}
	\left\|u\left(\cdot,t_0\right)\right\|^\sigma_{L^2(\omega)}\left\|u\left(\cdot,\frac{t_0}{2}\right)\right\|^{1-\sigma}_{H^1(\Omega)}.
\end{equation*}
	This, together with (\ref{yu-7-12-2}), implies (\ref{yu-7-12-1}) and completes the proof.
\end{proof}

\subsection{Proof of Theorem \ref{yu-main-1}}

To make the paper self-contained, we here provide the proof of Theorem \ref{yu-main-1} in detail, although 
it is almost the same as the proof of \cite[Theorem 1.1]{Phung-Wang-2013} or \cite[Theorem 1]{Apraiz-Escauriaza-Wang-Zhang}.

\begin{lemma}\label{yu-lemma-7-13-1}
    (\cite[Proposition 2.1]{Phung-Wang-2013})	Let $E\subset(0,T)$ be a measurable set of positive measure, 
	$\ell$ be a density point of $E$. Then for each $z>1$, there exists  $\ell_1\in(\ell,T)$
	such that $\{\ell_m\}_{m\in\mathbb{N}^+}$ given by
\begin{equation}\label{yu-7-13-b-1}
	\ell_{m+1}=\ell+\frac{1}{z^m}(\ell_1-\ell)
\end{equation}
verifies 
\begin{equation}\label{yu-7-13-b-2}
	\ell_m-\ell_{m+1}\leq 3|E\cap(\ell_{m+1},\ell_m)|.
\end{equation}	 
\end{lemma}

\medskip

\begin{proof}[\textbf{Proof of Theorem \ref{yu-main-1}}]

By (\ref{yu-7-12-1}), one can show that, for arbitrary fixed $\epsilon>0$ and any $t_0\in(0,T)$, 
\begin{equation*}\label{yu-7-13-5}
	\left\|u\left(\cdot,t_0\right)\right\|_{L^2(\Omega)}
	\leq \frac{Ce^{\frac{C(T^2+1)}{t_0}}}{\epsilon^{\gamma}}
	\left\|u\left(\cdot,t_0\right)\right\|_{L^2(\omega)}
	+\epsilon\|u(\cdot,0)\|_{L^2(\Omega)},
\end{equation*}
	where $\gamma>0$ is a constant. By a translation in time, one has for each $0\leq t_1<t_2<T$, 
\begin{equation*}\label{yu-7-13-6}
	\left\|u\left(\cdot,t_2\right)\right\|_{L^2(\Omega)}
	\leq \frac{Ce^{\frac{C(T^2+1)}{t_2-t_1}}}{\epsilon^{\gamma}}
	\left\|u\left(\cdot,t_2\right)\right\|_{L^2(\omega)}
	+\epsilon\|u(\cdot,t_1)\|_{L^2(\Omega)}\;\;\mbox{for all}\;\;\epsilon>0.
\end{equation*}
	Let $0<\ell_{m+2}<\ell_{m+1}\leq t<\ell_m<T$, by (\ref{yu-7-13-6}), we get 
\begin{equation}\label{yu-7-13-7}
	\|u(\cdot,t)\|_{L^2(\Omega)}
	\leq \frac{Ce^{\frac{C(T^2+1)}{t-\ell_{m+2}}}}{\epsilon^\gamma}\|u(\cdot,t)\|_{L^2(\omega)}
	+\epsilon\|u(\cdot,\ell_{m+2})\|_{L^2(\Omega)}\;\;\mbox{for all}\;\;\epsilon>0.
\end{equation}
	Noting that, by (\ref{yu-7-12-2-b}), 
\begin{equation*}\label{yu-7-13-8}
	e^{-CT}\|u(\cdot,\ell_m)\|_{L^2(\Omega)}\leq \|u(\cdot,t)\|_{L^2(\Omega)}. 
\end{equation*}
	This, along with (\ref{yu-7-13-7}), yields that for any $\epsilon>0$,
\begin{equation*}\label{yu-7-13-9}
	\|u(\cdot,\ell_m)\|_{L^2(\Omega)}
	\leq \frac{Ce^{\frac{C(T^2+1)}{\ell_{m+1}-\ell_{m+2}}}}{\epsilon^\gamma}
	\|u(\cdot,t)\|_{L^2(\omega)}+\epsilon\|u(\cdot,\ell_{m+2})\|_{L^2(\Omega)}.
\end{equation*}
	Integrating over $E\cap(\ell_{m+1},\ell_{m})$, we get 
\begin{eqnarray*}\label{yu-7-13-10}
	\|u(\cdot,\ell_m)\|_{L^2(\Omega)}
	\leq\frac{Ce^{\frac{C(T^2+1)}{\ell_{m+1}-\ell_{m+2}}}}{|E\cap (\ell_{m+1},\ell_{m})|\epsilon^\gamma}
	\int_{\ell_{m+1}}^{\ell_m}\chi_E\|u(\cdot,t)\|_{L^2(\omega)}dt
	+\epsilon \|u(\cdot,\ell_{m+2})\|_{L^2(\Omega)}.
\end{eqnarray*}
	This, together with (\ref{yu-7-13-b-1}) and (\ref{yu-7-13-b-2}), gives 
\begin{eqnarray*}\label{yu-7-13-11}
	\epsilon^\gamma e^{-\eta z^{m+2}}\|u(\cdot,\ell_m)\|_{L^2(\Omega)}
	-\epsilon^{1+\gamma}e^{-\eta z^{m+2}}\|u(\cdot,\ell_{m+2})\|_{L^2(\Omega)}
	\leq C\int_{\ell_{m+1}}^{\ell_m}\chi_E\|u(\cdot,t)\|_{L^2(\omega)}dt,
\end{eqnarray*}
	where $\eta:=\frac{C(T^2+1)}{z(z-1)(\ell_1-\ell)}$. 
	Letting $\epsilon:=e^{-\eta z^{m+2}}$ and $z:=\sqrt{\frac{2+\gamma}{1+\gamma}}$, we have 
\begin{eqnarray}\label{yu-7-13-12}
	e^{-\eta(2+\gamma)z^m}\|u(\cdot,\ell_m)\|_{L^2(\Omega)}
	-e^{-\eta(2+\gamma)z^{m+2}}\|u(\cdot,\ell_{m+2})\|_{L^2(\Omega)}
	\leq C\int_{\ell_{m+1}}^{\ell_m}\chi_E\|u(\cdot,t)\|_{L^2(\omega)}dt.
\end{eqnarray}
	By taking first  $m=2m'$ and then summing the estimate (\ref{yu-7-13-12}) from 
	$m'=1$ to infinity,  we obatin
\begin{eqnarray*}\label{yu-7-14-1}
	&\;&\sum_{m'=1}^\infty\left[e^{-\eta(2+\gamma)z^{2m'}}\|u(\cdot,\ell_{2m'})\|_{L^2(\Omega)}
	-e^{-\eta(2+\gamma)z^{2m'+2}}\|u(\cdot,\ell_{2m'+2})\|_{L^2(\Omega)}\right]\nonumber\\
	&\leq& C\sum_{m'=1}^\infty\int_{E\cap(\ell_{2m'+1},\ell_{2m'})}\|u(\cdot,t)\|_{L^2(\omega)}dt
	\leq C\int_{E}\|u(\cdot,t)\|_{L^2(\omega)}dt.
\end{eqnarray*}
	Note that $e^{-\eta(2+\gamma)z^{2m'}}\to 0$ as $m'\to \infty$. Therefore, 
\begin{equation*}\label{yu-7-14-2}
	\|u(\cdot,\ell_2)\|_{L^2(\Omega)}\leq Ce^{\eta(2+\gamma)z^2}
	\int_{E}\|u(\cdot,t)\|_{L^2(\omega)}dt.
\end{equation*}
	This, along with (\ref{yu-7-12-2-b}), leads to the desired observability inequality. 
\par
	Finally, when $E=[0,T]$, we can take $\ell=0$ and $\ell_1=T$ in the above argument to conclude the desired result. 
\end{proof}

\section{Proofs of quantitative estimates of unique continuation}\label{dujin1}

\subsection{Preliminary lemmas}

\subsubsection{Local energy estimates and exponential decay}

Suppose $\rho\in(0,\min\{1,\rho_0\})$ such that $\Omega_\rho\neq \emptyset$, $T>0$, $t_0\in(0,T)$ and $x_0\in\Omega_\rho$.   Let $u\in C([0,2T];H^1(\triangle_{ \rho}(x_0)))$ be a solution of 
\begin{equation}\label{yu-11-29-3}
\begin{cases}
    l(x)u_t-\mbox{div}(A(x)\nabla u)+b(x)u=0&\mbox{in}\;\;\triangle_{\rho}(x_0)\times(0,2T),\\
    u(\cdot,0)=0&\mbox{in}\;\;\triangle_{ \rho}(x_0).
\end{cases}
\end{equation}
Assume $\eta\in C^\infty(\mathbb{R}^+;[0,1])$ is a cutoff function satisfying  
\begin{equation}\label{yu-6-6-6}
\begin{cases}
	\eta\equiv 1 &\mbox{in}\;\;(0,t_0),\\
	\eta\equiv0 &\mbox{in}\;\; [T,+\infty),\\
	|\eta_t|\leq \frac{C}{T-t_0}&\mbox{in}\;\;(t_0,T)
\end{cases}
\end{equation}
with a generic positive constant $C$ independent  of $t_0$ and $T$. Set
\begin{equation}\label{du8141}
	R_0=
\begin{cases}
	\Theta_N^{-N}\left(8\sqrt{2}\Lambda_1\Lambda_2\Gamma_2(\triangle_1(0),N,\frac{N}{2})\right)^{-\frac{N+\delta}{\delta}}
	&\mbox{if (i) in (\ref{yu-6-24-1-1-b}) holds},\\
	\frac{N(N-2)}{4\sqrt{2}\Lambda_1\Lambda_2}&\mbox{if (ii) in (\ref{yu-6-24-1-1-b}) holds},
\end{cases}
\end{equation} 
and take $R\in(0,\min\{R_0,\rho\})$. Here $\Theta_N=|\triangle_1(0)|$. 
Let $v$ be the solution of
    \begin{equation}\label{yu-11-29-4}
\begin{cases}
    l(x)v_t-\mbox{div}(A(x)\nabla v)+b(x)v=0&\mbox{in}\;\;\triangle_{R}(x_0)\times\mathbb{R}^+,\\
    v=\eta u&\mbox{on}\;\;\partial\triangle_R(x_0)\times\mathbb{R}^+,\\
    v(\cdot,0)=0&\mbox{in}\;\;\triangle_R(x_0),
\end{cases}
\end{equation}
where $u$ satisfies  \eqref{yu-11-29-3} and $\eta$ verifies \eqref{yu-6-6-6}. Then, we have
the following exponential decay estimate of  $H^1$-energy for \eqref{yu-11-29-4}.
    
\begin{lemma}\label{yu-lemma-6-10-1}
      There exists
a generic constant $C>0$ such that
\begin{equation*}\label{yu-6-18-1}
 	\|v(\cdot,t)\|_{H^1(\triangle_R(x_0))}^2\leq CT^{-1}e^{CR^{1-N}T\left(1+\frac{1}{T-t_0}\right)-CR^{-2}(t-T)^+}F^2(R)\quad\text{for all}\;\; t\in\mathbb{R}^+,
 \end{equation*}
where $(t-T)^+=\max\{0,t-T\}$ and
	$F(R)=\sup_{s\in[0,T]}\|u(\cdot,s)\|_{H^1(\triangle_R(x_0))}$.
\end{lemma}

\begin{proof}%[\textbf{Proof of Lemma \ref{yu-lemma-6-10-1}}]
We proceed the proof into two steps as follows.
 
\par
\vskip 5pt
    \textbf{Step 1. To prove \eqref{yu-6-18-1} when $t\in[0,T]$}. 
Setting $w=v-\eta u$ in $\triangle_R(x_0)\times\mathbb R^+$, we find that  $w$ verifies that 
\begin{equation}\label{yu-6-6-7}
\begin{cases}
	l(x)w_t-\mbox{div}(A(x)\nabla w)+b(x)w=-l(x)\eta_tu
	&\mbox{in}\;\;\triangle_R(x_0)\times\mathbb R^+,\\
	w=0&\mbox{on}\;\;\partial\triangle_R(x_0)\times\mathbb R^+,\\
	w(\cdot,0)=0&\mbox{in}\;\;\triangle_R(x_0).
\end{cases}
\end{equation}
	We now prove that for each $t\in[0,T]$, 
\begin{equation}\label{yu-6-7-6}
	\|w(\cdot,t)\|^2_{L^2(\triangle_R(x_0))}\leq \frac{CT}{T-t_0}e^{C\left(R^{1-N}+\frac{1}{T-t_0}\right)T}
	\sup_{s\in[0,T]}\|u(\cdot,s)\|^2_{L^2(\triangle_R(x_0))}
\end{equation}
with a generic constant $C>0$. We divide our proof into two cases. 
\vskip 5pt
	\emph{Case I. $|b(x)|\leq \frac{\Lambda_2}{|x|}$ a.e. $x\in\Omega$.}
	Indeed, 
multiplying  first (\ref{yu-6-6-7}) by $w$ and integrating by parts over $\triangle_R(x_0)\times(0,t)$, 
along with the Hardy inequality in Lemma \ref{hardy}, we have 
\begin{eqnarray*}\label{yu-6-7-1}
	&\;&\frac{1}{2}\int_{\triangle_R(x_0)}l|w(\cdot,t)|^2dx+\int_0^t\int_{\triangle_R(x_0)}\nabla w\cdot(A\nabla w)dxds\nonumber\\
	&\leq&\int_0^t\int_{\triangle_R}|b||w|^2dxds
	+\int_0^t\int_{\triangle_R(x_0)}l\eta_tuwdxds\nonumber\\
	&\leq&\Lambda_2\int_0^t\int_{\triangle_{R}(x_0)}|x|^{-1}|w|^2dxds
	+\frac{1}{2}\int_0^t\int_{\triangle_R(x_0)}l|\eta_t||u|^2dxds
	+\frac{1}{2}\int_0^t\int_{\triangle_R(x_0)}l|\eta_t||w|^2dxds\nonumber\\
	&\leq&\frac{1}{2}\epsilon
	\int_0^t\int_{\triangle_R(x_0)}|x|^{-2}|w|^2dxds
	+\frac{CT}{T-t_0}\sup_{s\in[0,T]}\|u(\cdot,s)\|_{L^2(\triangle_R(x_0))}^2\nonumber\\
	&\;&+\frac{1}{2}\left(\epsilon^{-1}\Lambda_2^2+\frac{C}{T-t_0}\right)
	\int_0^t\int_{\triangle_R(x_0)}|w|^2dxds\nonumber\\
	&\leq&\frac{2\epsilon\Lambda_1}{(N-2)^2}
	\int_0^t\int_{\triangle_R(x_0)}\nabla w\cdot(A\nabla w)dxds
	+\frac{CT}{T-t_0}\sup_{s\in[0,T]}\|u(\cdot,s)\|_{L^2(\triangle_R(x_0))}^2\nonumber\\
	&\;&+\frac{1}{2}\left(\epsilon^{-1}\Lambda_2^2+\frac{C}{T-t_0}\right)
	\int_0^t\int_{\triangle_R(x_0)}|w|^2dxds.
\end{eqnarray*}
	Taking $\epsilon=\frac{(N-2)^2}{2\Lambda_1}$
    	in the above inequality,  we obtain 
\begin{eqnarray*}\label{yu-6-7-3}
	\|w(\cdot,t)\|_{L^2(\triangle_R(x_0))}^2\leq
	C\left(1+\frac{1}{T-t_0}\right)\int_0^t\|w(\cdot,s)\|^2_{L^2(\triangle_R(x_0))}ds
	+\frac{CT}{T-t_0}\sup_{s\in[0,T]}\|u(\cdot,s)\|^2_{L^2(\triangle_R(x_0))}
\end{eqnarray*}
 	with a generic constant $C>0$. By the Gronwall inequality, we get (\ref{yu-6-7-6}) immediately. 
\vskip 5pt
\emph{Case II. $b(\cdot)\in L^{N+\delta}(\Omega)$ and $\|b(\cdot)\|_{L^{N+\delta}(\Omega)}\leq\Lambda_2$.} We first note that, by using a standard scaling technique
to (\ref{yu-9-26-2}), without lose of generality, one has
\begin{equation}\label{yu-9-29-1}
	\Gamma_2(\triangle_r(x_0),N,\eta)=\Gamma_2(\triangle_1(0),N,\eta) r^{-\frac{2N}{N+2\eta}}
	\;\;\mbox{for each}\;\;r\in(0,1).
\end{equation}
	Multiplying first (\ref{yu-6-6-7}) by $w$ and then integrating by parts over $\triangle_R(x_0)\times(0,t)$, along with (\ref{yu-9-26-2}) (by letting $\epsilon=\frac{1}{2}$ there) and (\ref{yu-9-29-1}), we have  
        \begin{eqnarray*}\label{yu-6-7-1}
	&\;&\frac{1}{2}\int_{\triangle_R(x_0)}l(x)|w(x,t)|^2dx+\int_0^t\int_{\triangle_R(x_0)}\nabla w\cdot(A\nabla w)dxds\nonumber\\
	&\leq&CR^{-\frac{2N}{N+1}}\|b\|_{L^{\frac{N+1}{2}}(\triangle_R(x_0))}\int_0^t
	\|w\|^{\frac{2}{N+1}}_{L^2(\triangle_R(x_0))}\|w\|^{\frac{2N}{N+1}}_{H^1(\triangle_R(x_0))}ds
	+\int_0^t\int_{\triangle_R(x_0)}l\eta_tuwdxds\nonumber\\
	&\leq&C\epsilon^{-N}R^{N+1}\|b\|_{L^{N}(\triangle_R(x_0))}^{N+1}\int_0^t\|w\|^2_{L^2(\triangle_R(x_0))}ds
	+\epsilon R^{-2}\int_0^t\|w\|_{H^1(\triangle_R(x_0))}^2ds\nonumber\\
	&\;&+\frac{1}{2}\int_0^t\int_{\triangle_R(x_0)}l|\eta_t||u|^2dxds
	+\frac{1}{2}\int_0^t\int_{\triangle_R(x_0)}l|\eta_t||w|^2dxds\nonumber\\
	&\leq&\left(C\epsilon^{-N}R^{N+1}\Lambda_2^{N+1}+\frac{C}{2(T-t_0)}
	+\epsilon R^{-2}\right)\int_0^t\|w\|^2_{L^2(\triangle_R(x_0))}ds\nonumber\\
	&\;&+\epsilon R^{-2}\Lambda_1\int_0^t\int_{\triangle_R(x_0)}\nabla w\cdot (A\nabla w)dxds
	+\frac{CT}{T-t_0}\sup_{s\in[0,T]}\|u(\cdot,s)\|^2_{L^2(\triangle_R(x_0))}
	\qquad \text{for any}\;\;\epsilon>0.
\end{eqnarray*}
	 Taking
	$\epsilon=\frac{R^2}{2\Lambda_1}$
    	in the above inequality,  we  obtain 
\begin{multline*}\label{yu-6-7-3}
	\|w(\cdot,t)\|_{L^2(\triangle_R(x_0))}^2\leq
	C\left(R^{1-N}+\frac{1}{T-t_0}\right)\int_0^t\|w(\cdot,s)\|^2_{L^2(\triangle_R(x_0))}ds\\
	+\frac{CT}{T-t_0}\sup_{s\in[0,T]}\|u(\cdot,s)\|^2_{L^2(\triangle_R(x_0))}
\end{multline*}
with a generic constant $C>0$.
By the Gronwall inequality, we get (\ref{yu-6-7-6}) immediately. 
Hence, from (\ref{yu-6-7-6}) and the definition of $w$, we know that for any $t\in[0,T]$,
\begin{equation}\label{yu-6-7-11-1}
	\|v(\cdot,t)\|^2_{L^2(\triangle_R(x_0))}\leq C\left(1+\frac{T}{T-t_0}\right)e^{C\left(R^{1-N}+\frac{1}{T-t_0}\right)T}
	\sup_{s\in[0,T]}\|u(\cdot,s)\|^2_{L^2(\triangle_R(x_0))}
\end{equation}
with a generic constant $C>0$.
\par
	Next, we show that 
\begin{eqnarray}\label{yu-6-8-3}
	\|\nabla w(\cdot,t)\|_{L^2(\triangle_R(x_0))}^2
	\leq \left(1+\frac{1}{T-t_0}\right)\frac{CT}{T-t_0}e^{CR^{1-N}T\left(1+\frac{1}{T-t_0}\right)}
	\sup_{s\in[0,T]}\|u(\cdot,s)\|^2_{L^2(\triangle_R(x_0))}.
\end{eqnarray}
Which, along  with the definition of $w$, gives that for each $t\in[0,T]$,  
\begin{multline}\label{yu-6-8-4}
	\|\nabla v(\cdot,t)\|_{L^2(\triangle_R(x_0))}^2\leq 2\|\nabla w(\cdot,t)\|^2_{L^2(\triangle_R(x_0))}+2\|\nabla u(\cdot,t)\|_{L^2(\triangle_R(x_0))}^2\\
	\leq \left(1+\frac{1}{T-t_0}\right)\frac{CT}{T-t_0}e^{C\left(R^{1-N}+\frac{1}{T-t_0}\right)T}
	\sup_{s\in[0,T]}\|u(\cdot,s)\|^2_{H^1(\triangle_R(x_0))}
\end{multline}
with a generic constant $C>0$.
Hence,  the desired estimate \eqref{yu-6-18-1} follows from  (\ref{yu-6-7-11-1}) and (\ref{yu-6-8-4})
when $t\in[0,T]$. We also divide the proof of (\ref{yu-6-8-3}) into two cases under the assumptions in 
(\ref{yu-6-24-1-1-b}). 
\vskip 5pt
	\emph{Case I.  $|b(x)|\leq \frac{\Lambda_2}{|x|}$ a.e. $x\in\Omega$.}
Multiplying first (\ref{yu-6-6-7}) by $w_t$ and then integrating by parts  over $\triangle_R(x_0)\times(0,t)$,  
	we find 
\begin{eqnarray*}\label{yu-6-7-12}
	&\;&\int_0^t\int_{\triangle_R(x_0)}l|w_t|^2dxds
	+\frac{1}{2}\int_0^t\int_{\triangle_R(x_0)}[\nabla w\cdot(A\nabla w)]_tdxds\nonumber\\
	&\leq&\frac{1}{2}\epsilon\int_0^t\int_{\triangle_R(x_0)}
	|x|^{-2}|w|^2dxds+\frac{\Lambda_2^2\Lambda_3
	+\frac{C}{T-t_0}}{2\epsilon}\int_0^t\int_{\triangle_R(x_0)}l(x)|w_t|^2dxds\nonumber\\
	&\;&+\frac{C\epsilon}{2(T-t_0)}\int_0^t\int_{\triangle_R(x_0)}|u|^2dxds\qquad \text{for any}\;\;\epsilon>0
\end{eqnarray*}
with a generic constant $C>0$.
	Letting
	$\epsilon=\frac{\Lambda_2^2\Lambda_3+\frac{C}{T-t_0}}{2}$
	in the inequality above, combined with the Hardy inequality in Lemma \ref{hardy}, leads to 
\begin{eqnarray*}\label{yu-6-8-1}
	&\;&\int_{\triangle_R(x_0)}\nabla w(x,t)\cdot(A(x)\nabla w(x,t))dx\nonumber\\
	&\leq&C\left(1+\frac{1}{T-t_0}\right)
	\int_0^t\int_{\triangle_R(x_0)}|\nabla w|^2dxds+\left(1+\frac{1}{T-t_0}\right)\frac{CT}{T-t_0}
	\sup_{s\in[0,T]}\|u(\cdot,s)\|^2_{L^2(\triangle_R(x_0))},
\end{eqnarray*}	
for a generic constant $C>0$.
	This, together with the uniform ellipticity condition \eqref{yu-11-28-2}, means that 
\begin{eqnarray*}\label{yu-6-8-2}
	\|\nabla w(\cdot,t)\|_{L^2(\triangle_R(x_0))}^2&\leq&
	C\left(1+\frac{1}{T-t_0}\right)
	\int_0^t\|\nabla w(\cdot,s)\|^2_{L^2(\triangle_R(x_0))}ds\nonumber\\
	&\;&+\left(1+\frac{1}{T-t_0}\right)\frac{CT}{T-t_0}
	\sup_{s\in[0,T]}\|u(\cdot,s)\|^2_{L^2(\triangle_R(x_0))}.
\end{eqnarray*}
By the Gronwall inequality, we get (\ref{yu-6-8-3}). 
\vskip 5pt
\emph{Case II. $b(\cdot)\in L^{N+\delta}(\Omega)$ and $\|b(\cdot)\|_{L^{N+\delta}(\Omega)}\leq\Lambda_2$.} 
	Multiplying (\ref{yu-6-6-7}) by $w_t$ and integrating by parts over 
	$\triangle_R(x_0)\times (0,t)$, we have 
\begin{eqnarray*}\label{yu-9-30-1}
	&\;&\int_0^t\int_{\triangle_R(x_0)}l|w_t|^2dxds+\frac{1}{2}\int_0^t\int_{\triangle_R(x_0)}
	[\nabla w\cdot (A\nabla w)]_tdxds\nonumber\\
	&\leq&\frac{\epsilon}{2}\int_0^t\int_{\triangle_R(x_0)}|b|^2|w|^2dxds+\frac{\Lambda_3+\frac{C}{T-t_0}}{2\epsilon}
	\int_0^t\int_{\triangle_R(x_0)}l|w_t|^2dxds\nonumber\\
	&\;&+\frac{C\epsilon}{2(T-t_0)}\int_0^t\int_{\triangle_R(x_0)}|u|^2dxds\nonumber\\
	&\leq&\frac{C\epsilon R^{-2}\Lambda_2^{\frac{1}{2}}}{2}\int_0^t\int_{\triangle_R(x_0))}|w|^2dxds
	+\frac{C\epsilon R^{-2}\Lambda_1\Lambda_2^{\frac{1}{2}}}{2}
	\int_0^t\int_{\triangle_R(x_0)}\nabla w\cdot (A\nabla w)dxdt\nonumber\\
	&\;&+\frac{\Lambda_3+\frac{C}{T-t_0}}{2\epsilon}\int_0^t\int_{\triangle_R(x_0)}l|w_t|^2dxds
	+\frac{C\epsilon}{2(T-t_0)}\int_0^t\int_{\triangle_R(x_0)}|u|^2dxds. 
\end{eqnarray*}
	Here, we used (\ref{yu-9-26-2}) and (\ref{yu-9-29-1}). Taking $\epsilon=\frac{\Lambda_3+\frac{C}{T-t_0}}{2}$ in the above inequality, 
	by (\ref{yu-6-7-6}) we get 
\begin{eqnarray*}\label{yu-9-30-2}
	&\;&\int_{\triangle_R(x_0)}\nabla w(x,t)\cdot(A\nabla w(x,t))dx\nonumber\\
	&\leq&CR^{-2}\left(1+\frac{1}{T-t_0}\right)\int_0^t\int_{\triangle_R(x_0)}(|w|^2+\nabla w\cdot (A\nabla w))dxds\nonumber\\
	&\;&+\frac{CT}{T-t_0}\left(1+\frac{1}{T-t_0}\right)\sup_{s\in[0,T]}\|u(\cdot,s)\|_{L^2(\triangle_R(x_0)}^2\nonumber\\
	&\leq&CR^{1-N}\left(1+\frac{1}{T-t_0}\right)\int_0^t\int_{\triangle_R(x_0)}\nabla w\cdot (A\nabla w)dxds\nonumber\\
	&\;&+C\frac{T}{T-t_0}\left(1+\frac{1}{T-t_0}\right)e^{C\left(R^{1-N}+\frac{1}{T-t_0}\right)T}
	\sup_{s\in[0,T]}\|u(\cdot,s)\|_{L^2(\Omega)}^2. 
\end{eqnarray*}
	By the Gronwall inequality, we get (\ref{yu-6-8-3}). 

\medskip

\vskip 5pt
 \textbf{Step 2. To prove \eqref{yu-6-18-1} when $t\geq T$.}
 \vskip 5pt
 
Define for any  $f,g\in C_0^\infty(\triangle_R(x_0))$,
$$
\langle f,g\rangle_{\mathcal L^2(\triangle_R(x_0))} :=\int_{\triangle_R(x_0)}l(x)f(x)g(x)dx \quad \text{and}\quad
\|f\|_{\mathcal L^2(\triangle_R(x_0)}:=\langle f,f\rangle_{\mathcal L^2(\triangle_R(x_0))}^{1/2}.$$
Set
$\mathcal{L}^2(\triangle_{R}(x_0))=\overline{C_0^\infty(\triangle_R(x_0))}^{\|\cdot\|_{\mathcal L^2(\triangle_R(x_0))}}
$.
Since $l$ is positive, it is clear that $\mathcal{L}^2(\triangle_R(x_0))=L^2(\triangle_R(x_0))$ with an equivalent norm. 
 Denoting $\mathcal{A}=-l^{-1}[\mbox{div}(A\nabla)-b]$,
we claim that 
	there is a generic constant $C>0$ (independent of $R$) such that  
 \begin{equation}\label{yu-6-7-9}
 	\langle \mathcal{A}f,f\rangle_{\mathcal{L}^2(\triangle_R(x_0))}\geq CR^{-2}\|f\|^2_{L^2(\triangle_R(x_0))}
	\;\;\mbox{for each}\;\;f\in H_{0}^1(\triangle_R(x_0))\cap H^2(\triangle_R(x_0)).
 \end{equation}
 	We also divide its proof into two cases. 
\vskip 5pt
        \emph{Case I.  $|b(x)|\leq \frac{\Lambda_2}{|x|}$ a.e. $x\in\Omega$.}
        In this case, by Lemma \ref{hardy}, we find that for each $\epsilon>0$, 
$$
	\langle \mathcal{A} f,f\rangle_{\mathcal{L}(\triangle_R(x_0))}\geq \Lambda_1^{-1}\int_{\triangle_{R}(x_0)}
	|\nabla f|^2dx-\frac{2\epsilon}{(N-2)^2}\int_{\triangle_R(x_0)}|\nabla f|^2dx-\frac{\Lambda_2^2}{2\epsilon}
	\int_{\triangle_R(x_0)}|f|^2dx,
$$
       for any $f\in H_0^1(\triangle_R(x_0))\cap H^2(\triangle_R(x_0))$.  Letting 
       $\epsilon=\frac{(N-2)^2}{4\Lambda_1}$ in the above inequality, by 
        the Poincar\'e inequality 
\begin{equation}\label{yu-11-30-b-1}
    \int_{\triangle_R(x_0)}|f(x)|^2dx\leq \left(\frac{2R}{N}\right)^2\int_{\triangle_R(x_0)}|\nabla f(x)|^2dx\;\;\mbox{for each}\;\;f\in H_0^1(\triangle_R(x_0)),
\end{equation}	 
         we derive 
\begin{equation*}\label{yu-10-12-1}
	\langle\mathcal{A}f,f\rangle_{\mathcal{L}^2(\triangle_R(x_0))}\geq
	\left[\frac{\Lambda_1^{-1}}{2}-\frac{8\Lambda_1\Lambda_2^2R^2}{(N-2)^2N^2}\right]\int_{\triangle_R(x_0)}
	|\nabla f|^2dx. 
\end{equation*}
 	From the definition of $R_0$ given in (\ref{du8141}), and (\ref{yu-11-30-b-1}), we can conclude
	the claim (\ref{yu-6-7-9}). 
\vskip 5pt
\emph{Case II. $b(\cdot)\in L^{N+\delta}(\Omega)$ and $\|b(\cdot)\|_{L^{N+\delta}(\Omega)}\leq \Lambda_2$.} 	
	By using (\ref{yu-9-26-2}), (\ref{yu-9-29-1}) and (\ref{yu-11-30-b-1}), we have 
\begin{eqnarray*}\label{yu-10-12-2}
	\int_{\triangle_R(x_0)}|b||f|^2dx&\leq& \Gamma_2\left(\triangle_R(x_0),N,\frac{N}{2}\right)
	\|b\|_{L^{N}(\triangle_R(x_0))}\|f\|_{L^2(\triangle_R(x_0))}
	\|f\|_{H^1_0(\triangle_R(x_0))}\nonumber\\
	&\leq&\sqrt{2}\Gamma_2\left(\triangle_1(0),N,\frac{N}{2}\right)\|b\|_{L^{N}(\triangle_R(x_0))}\|\nabla f\|_{L^2(\triangle_R(x_0))}^2\nonumber\\
	&\leq&\sqrt{2}\Theta_N^{\frac{N\delta}{N+\delta}}\Gamma_2\left(\triangle_1(0),N,\frac{N}{2}\right)R^{\frac{\delta}{N+\delta}}\|b\|_{L^{N+\delta}(\Omega)}\|\nabla f\|_{L^2(\triangle_R(x_0))}^2.
\end{eqnarray*}
	From the definition of $R_0$, we have 
$$
	\int_{\triangle_R(x_0)}|b||f|^2dx\leq \frac{\Lambda_1^{-1}}{8}\|\nabla f\|^2_{L^2(\triangle_R(x_0))}.
$$
	This implies 
$$
	\langle \mathcal{A}f, f\rangle_{\mathcal{L}^2(\triangle_R(x_0))}\geq \frac{7\Lambda_1^{-1}}{8}\int_{\triangle_R(x_0)}|\nabla f|^2dx.
$$
	and then (\ref{yu-6-7-9}) holds.

    As a consequence of \eqref{yu-6-7-9}, we see that the inverse of $\mathcal A$ is positive, self-adjoint and compact in $\mathcal L^2(\triangle_R(x_0))$.
By the spectral theorem for compact self-adjoint operators,	there are  eigenvalues 
	$\{\mu_i\}_{i\in\mathbb{N}^+}\subset \mathbb{R}^+$ and eigenfunctions  
	$\{f_i\}_{i\in\mathbb{N}^+}\subset H_0^1(\triangle_R(x_0))$, which make up an orthogonal basis of $\mathcal{L}^2(\triangle_R(x_0))$,
such that 
 \begin{equation}\label{yu-6-7-10}
 \begin{cases}
 	-\mathcal{A}f_i=\mu_if_i\;\;\mbox{and}\;\;\|f_i\|_{\mathcal{L}^2(\triangle_R(x_0))}=1&\mbox{for each}\;\;i\in\mathbb{N}^+,\\
CR^{-2}< \mu_1\leq \mu_2\leq \cdots
         \leq \mu_i\to+\infty&\mbox{as}\;\;i\to+\infty.
 \end{cases}
 \end{equation}
 Then, by the formula of Fourier decomposition, 
	the solution $w$  of (\ref{yu-6-6-7}) in $[T,+\infty)$ is given by
\begin{equation*}\label{yu-6-12-3}
	w(\cdot,t)=\sum_{i=1}^\infty\langle w(\cdot,T),f_i\rangle_{\mathcal{L}^2(\triangle_R(x_0))}e^{-\mu_i(t-T)}f_i	\quad\text{in}\;\;\triangle_R(x_0)\;\;\mbox{for each}\;\;t\in[T,+\infty).
\end{equation*}
    	Hence, we deduce that for each $t\in[T,+\infty)$,
\begin{eqnarray}\label{yu-6-12-4}
	\|w(\cdot,t)\|^2_{\mathcal{L}^2(\triangle_R(x_0))}
	\leq e^{-CR^{-2}(t-T)}\|w(\cdot,T)\|_{\mathcal{L}^2(\triangle_R(x_0))}^2
\end{eqnarray}
   and
\begin{eqnarray*}%\label{yu-6-12-5}
	w_t(\cdot,t)=-\sum_{i=1}^\infty\mu_i\langle w(\cdot,T),f_i\rangle_{\mathcal{L}^2(\triangle_R(x_0))}e^{-\mu_i(t-T)}f_i.
\end{eqnarray*}
        It follows that
\begin{eqnarray}\label{yu-6-13-1}
	-\langle w(\cdot,t),w_t(\cdot,t)\rangle_{\mathcal{L}^2(\triangle_R(x_0))}
	=\sum_{i\in\mathbb{N}^+}\mu_i|\langle w(\cdot,T),f_i\rangle_{\mathcal{L}^2(\triangle_R(x_0))}|^2
	e^{-2\mu_i(t-T)},
\end{eqnarray}
for each $t\in[T,+\infty)$. In particular, taking $t=T$ in the above identify leads to 
\begin{equation}\label{yu-6-13-2}
	-\langle w(\cdot,T),w_t(\cdot,T)\rangle_{\mathcal{L}^2(\triangle_R(x_0))}
	=\sum_{i\in\mathbb{N}^+}\mu_i|\langle w(\cdot,T),f_i\rangle_{\mathcal{L}^2(\triangle_R(x_0))}|^2.	
\end{equation} 
    Meanwhile, it follows from (\ref{yu-6-6-7}) and Lemma \ref{hardy} that 
\begin{eqnarray}\label{yu-6-13-3}
-\langle w(\cdot,T),w_t(\cdot,T)\rangle_{\mathcal{L}^2(\triangle_R(x_0))}\leq C\|w(\cdot,T)\|^2_{H_{0}^1(\triangle_R(x_0))}
\end{eqnarray}
with a generic constant $C>0$.
    	From (\ref{yu-6-13-2}) and (\ref{yu-6-13-3}), we have 
\begin{equation*}\label{yu-6-14-1}
	\sum_{i=1}^\infty\mu_i|\langle w(\cdot,T),f_i\rangle_{\mathcal{L}^2(\triangle_R(x_0))}|^2
	\leq C\|w(\cdot,T)\|^2_{H_{0}^1(\triangle_R(x_0))}.
\end{equation*}    
    	This, together with  (\ref{yu-6-13-1}), gives 
\begin{equation}\label{yu-6-14-2}
	-\langle w(\cdot,t),w_t(\cdot,t)\rangle_{\mathcal{L}^2(\triangle_R(x_0))}
	\leq Ce^{-CR^{-2}(t-T)}\|w(\cdot,T)\|^2_{H_{0}^1(\triangle_R(x_0))},
\end{equation}   
	for each $t\in[T,+\infty)$. On the other hand, by  (\ref{yu-6-6-7}) and (\ref{yu-6-7-9}), we see that for each $t\in[T,+\infty)$,
 \begin{equation}\label{yu-6-14-3}
 	-\langle w(\cdot,t),w_t(\cdot,t)\rangle_{\mathcal{L}^2(\triangle_R(x_0))}=\langle w(\cdot,t),-\mathcal{A}w(\cdot,t)\rangle_{\mathcal{L}^2(\triangle_R(x_0))}
	\geq C\|\nabla w(\cdot,t)\|^2_{L^2(\triangle_R(x_0))}.
 \end{equation}
    	By (\ref{yu-6-14-2}) and (\ref{yu-6-14-3}), we find that for each $t\in[T,\infty)$, 
 \begin{equation*}\label{yu-6-14-4}
 	\|\nabla w(\cdot,t)\|^2_{L^2(\triangle_R(x_0))}
	\leq C
	e^{-CR^{-2}(t-T)}\|w(\cdot,T)\|^2_{H_{0}^1(\triangle_R(x_0))} 
\end{equation*}
    	This, together with (\ref{yu-6-12-4}), means that 
\begin{equation*}\label{yu-6-18-2}
	\|w(\cdot,t)\|_{H_{0}^1(\triangle_R(x_0))}^2\leq Ce^{-CR^{-2}(t-T)}\|w(\cdot,T)\|^2_{H_{0}^1(\triangle_R(x_0))}.
\end{equation*}
By the fact that  $w(\cdot,t)=v(\cdot,t)$ for each $t\geq T$, we conclude the desired result.
\end{proof}
\medskip
We next define 
\begin{equation*}\label{yu-6-18-5}
	\tilde{v}(\cdot,t)=
\begin{cases}
	v(\cdot,t)&\mbox{if}\;\;t\geq 0,\\
	0&\mbox{if}\;\;t<0,
\end{cases}
\end{equation*}
where $v$ is the solution of \eqref{yu-11-29-4}.
By Lemma \ref{yu-lemma-6-10-1}, we can take the Fourier transform of	$\tilde{v}$ with respect to the time variable $t\in\mathbb R$
\begin{equation*}\label{yu-6-18-6}
	\hat{v}(x,\mu)=\int_{\mathbb{R}}e^{-i\mu t}\tilde{v}(x,t)dt\quad\text{for}\;\;(x,\mu)\in\triangle_R(x_0)\times\mathbb R.
\end{equation*}
Then, we have
\begin{lemma}\label{yu-lemma-6-18-1}
There exists a generic constant $C>0$ such that, for each  $\mu\in\mathbb{R}$,   
the following two estimates hold:
\begin{equation}\label{yu-6-23-5}
	\|\nabla \hat{v}(\cdot,\mu)\|_{L^2(\triangle_r(x_0))}\leq \frac{C(1+\sqrt{|\mu|})}{R-2r}\|\hat{v}(\cdot,\mu)\|_{L^2(\triangle_{\frac{R}{2}}(x_0))}\quad \text{for all} \;\;0<r <R/2,
\end{equation} 
and
\begin{equation}\label{yu-6-22-16}
	\|\hat{v}(\cdot,\mu)\|_{L^2(\triangle_{\frac{R}{2}}(x_0))}\leq CT^{-\frac{1}{2}}
	e^{CR^{1-N}\left(1+\frac{1}{T-t_0}\right)T-\frac{\sqrt{|\mu|}R}{4e\Pi}}F(R)
\end{equation}
with a positive constant $\Pi$. 
\end{lemma}
\begin{proof}
By \eqref{yu-11-29-4}, we have that for each $\mu\in\mathbb R$,
\begin{equation}\label{yu-6-18-7}
	i\mu l(x)\hat{v}(x,\mu)-\mbox{div}(A(x)\nabla \hat{v}(x,\mu))
	+b(x)\hat{v}(x,\mu)=0
	\;\;\mbox{in}\;\;\triangle_{R}(x_0). 
\end{equation}
Take arbitrarily $r\in(0,\frac{R}{2})$ and define a cutoff function $\psi \in C^\infty(\mathbb{R}^N;[0,1])$ verifying 
\begin{equation}\label{yu-6-23-1}
\begin{cases}
	\psi=1&\mbox{in}\;\;\overline{\triangle_r(x_0)},\\
	\psi=0&\mbox{in}\;\;\mathbb{R}^N\backslash \triangle_{\frac{R}{2}}(x_0),\\
	|\nabla\psi|\leq \frac{C}{R-2r}&\mbox{in}\;\;\mathbb{R}^N.
\end{cases}
\end{equation}
	Multiplying first (\ref{yu-6-18-7}) by $\bar{\hat{v}}\psi^2$ and then integrating by parts  over
	$\triangle_{\frac{R}{2}}(x_0)$, we have 
\begin{eqnarray*}\label{yu-6-23-2}
        \int_{\triangle_{\frac{R}{2}}(x_0)}\nabla\bar{\hat{v}}\cdot (A\nabla\hat{v})\psi^2dx
	+2\int_{\triangle_{\frac{R}{2}}(x_0)}\nabla\psi\cdot(A\nabla\hat{v})\bar{\hat{v}}\psi dx\nonumber\\
	=-i\int_{\triangle_{\frac{R}{2}}(x_0)}\mu l |\hat{v}|^2\psi^2dx
	-\int_{\triangle_{\frac{R}{2}}(x_0)}b|\hat{v}|^2\psi^2dx.
\end{eqnarray*}
	We divide the proof of (\ref{yu-6-23-5}) into two cases.
\vskip 5pt
	\emph{Case I.  $|b(x)|\leq \frac{\Lambda_2}{|x|}$ a.e. $x\in\Omega$.}
 	By (\ref{yu-11-28-2}), (\ref{yu-7-29-3}) and the Hardy inequality in Lemma \ref{hardy}, we derive that for each $\epsilon_1>0$ and $\epsilon_2>0$, 
\begin{eqnarray*}\label{yu-6-23-3}
	&\;&\Lambda_1^{-1}\int_{\triangle_{\frac{R}{2}}(x_0)}|\nabla\hat{v}|^2\psi^2dx\nonumber\\
	&\leq&2\Lambda_1\int_{\triangle_{\frac{R}{2}}(x_0)}|\nabla\hat{v}||\hat{v}||\nabla\psi||\psi|dx
	+\Lambda_3|\mu|\int_{\triangle_{\frac{R}{2}}(x_0)}|\hat{v}|^2\psi^2dx
	+\int_{\triangle_{\frac{R}{2}}(x_0)}|b||\hat{v}|^2\psi^2dx\nonumber\\
	&\leq&\epsilon_1\int_{\triangle_{\frac{R}{2}}(x_0)}|\nabla \hat{v}|^2\psi^2dx
	+\frac{\Lambda_1^2}{\epsilon_1}\int_{\triangle_{\frac{R}{2}}(x_0)}|\hat{v}|^2|\nabla\psi|^2dx
	+\epsilon_2\int_{\triangle_{\frac{R}{2}}(x_0)}|x|^{-2}|\hat{v}\psi|^2dx\nonumber\\
	&&\;\;\;\;+\left(\frac{\Lambda_2^2}{4\epsilon_2}+\Lambda_3|\mu|\right)\int_{\triangle_{\frac{R}{2}}(x_0)}|\hat{v}|^2\psi^2dx\nonumber\\
	&\leq&\left(\epsilon_1+\frac{8\epsilon_2}{(N-2)^2}\right)
	\int_{\triangle_{\frac{R}{2}}(x_0)}|\nabla\hat{v}|^2\psi^2dx
	+\left(\frac{8\epsilon_2}{(N-2)^2}+\frac{\Lambda_1^2}{\epsilon_1}\right)
	\int_{\triangle_{\frac{R}{2}}(x_0)}|\hat{v}|^2|\nabla \psi|^2dx\nonumber\\
	&\;&+\left(\frac{\Lambda_2^2}{4\epsilon_2}+\Lambda_3|\mu|\right)\int_{\triangle_{\frac{R}{2}}(x_0)}|\hat{v}|^2\psi^2dx.
\end{eqnarray*}
Taking $\epsilon_1=\frac{1}{4\Lambda_1}$ and $\epsilon_2=\frac{(N-2)^2}{16\Lambda_1}$ in the above inequality, we derive
 (\ref{yu-6-23-5}).
\vskip 5pt
\emph{Case II. $b(\cdot)\in L^{N+\delta}(\Omega)$ and $\|b(\cdot)\|_{L^{N+\delta}(\Omega)}\leq\Lambda_2$.}
	By (\ref{yu-11-28-2}), (\ref{yu-7-29-3}) and (\ref{yu-9-26-2}) with $\eta=\frac{N}{2}$, we get 
\begin{eqnarray*}\label{yu-6-23-3}
	&\;&\Lambda_1^{-1}\int_{\triangle_{\frac{R}{2}}(x_0)}|\nabla\hat{v}|^2\psi^2dx\nonumber\\
	&\leq&2\Lambda_1\int_{\triangle_{\frac{R}{2}}(x_0)}|\nabla\hat{v}||\hat{v}||\nabla\psi||\psi|dx
	+\Lambda_3|\mu|\int_{\triangle_{\frac{R}{2}}(x_0)}|\hat{v}|^2\psi^2dx
	+\int_{\triangle_{\frac{R}{2}}(x_0)}|b||\hat{v}|^2\psi^2dx\nonumber\\	
	&=&2\Lambda_1\int_{\triangle_{\frac{R}{2}}(x_0)}|\nabla\hat{v}||\hat{v}||\nabla\psi||\psi|dx
	+\Lambda_3|\mu|\int_{\triangle_{\frac{R}{2}}(x_0)}|\hat{v}|^2\psi^2dx\nonumber\\
	&\;&+\Gamma_2\left(\triangle_1(0),N,\frac{N}{2}\right)\|b\|_{L^{N}(\triangle_{\frac{R}{2}}(x_0))}
	\left(\frac{R}{2}\right)^{-1}\|\hat{v}\psi\|_{L^2(\triangle_{\frac{R}{2}}(x_0))}
	\|\hat{v}\psi\|_{H^1(\triangle_{\frac{R}{2}}(x_0))}.\nonumber\\
\end{eqnarray*}	
Then for any $\epsilon>0$, we have
\begin{eqnarray*}	
&\;&\Lambda_1^{-1}\int_{\triangle_{\frac{R}{2}}(x_0)}|\nabla\hat{v}|^2\psi^2dx\nonumber\\
	&\leq&\epsilon\int_{\triangle_{\frac{R}{2}}(x_0)}|\nabla \hat{v}|^2\psi^2dx
	+\frac{\Lambda_1^2}{\epsilon}\int_{\triangle_{\frac{R}{2}}(x_0)}|\hat{v}|^2|\nabla \psi|^2dx
	+\Lambda_3|\mu|\int_{\triangle_{\frac{R}{2}}(x_0)}|\hat{v}|^2\psi^2dx\nonumber\\
	&\;&+2\sqrt{2}\Gamma_2\left(\triangle_1(0),N,\frac{N}{2}\right)\|b\|_{L^N(\triangle_{\frac{R}{2}}(x_0))}
	\int_{\triangle_{\frac{R}{2}}(x_0)}\left(|\nabla \hat{v}|^2\psi^2
	+|\hat{v}|^2|\nabla\psi|^2\right)dx\nonumber\\
	&\leq&\epsilon\int_{\triangle_{\frac{R}{2}}(x_0)}|\nabla \hat{v}|^2\psi^2dx
	+\frac{\Lambda_1^2}{\epsilon}\int_{\triangle_{\frac{R}{2}}(x_0)}|\hat{v}|^2|\nabla \psi|^2dx
	+\Lambda_3|\mu|\int_{\triangle_{\frac{R}{2}}(x_0)}|\hat{v}|^2\psi^2dx\nonumber\\
	&\;&+2\sqrt{2}\Theta_N^{\frac{N\delta}{N+\delta}}\Gamma_2\left(\triangle_1(0),N,\frac{N}{2}\right)
	\Lambda_2R^{\frac{\delta}{N+\delta}}
	\int_{\triangle_{\frac{R}{2}}(x_0)}\left(|\nabla \hat{v}|^2\psi^2
	+|\hat{v}|^2|\nabla\psi|^2\right)dx\nonumber\\
	&\leq&\epsilon\int_{\triangle_{\frac{R}{2}}(x_0)}|\nabla \hat{v}|^2\psi^2dx
	+\frac{\Lambda_1^2}{\epsilon}\int_{\triangle_{\frac{R}{2}}(x_0)}|\hat{v}|^2|\nabla \psi|^2dx
	+\Lambda_3|\mu|\int_{\triangle_{\frac{R}{2}}(x_0)}|\hat{v}|^2\psi^2dx\nonumber\\
	&\;&+\frac{\Lambda_1^{-1}}{4}\int_{\triangle_{\frac{R}{2}}(x_0)}\left(|\nabla \hat{v}|^2\psi^2
	+|\hat{v}|^2|\nabla\psi|^2\right)dx.
\end{eqnarray*}	
	Here, we used (\ref{yu-9-29-1}) and the definition of $R_0$. Taking $\epsilon=\frac{\Lambda_1^{-1}}{4}$ in the above inequality and using (\ref{yu-6-23-1}) lead to (\ref{yu-6-23-5}).

\medskip

Note that, when $\mu=0$, by Lemma \ref{yu-lemma-6-10-1} we have
\begin{eqnarray}\label{yu-6-22-15}
	\|\hat{v}(\cdot,0)\|_{L^2(\triangle_{R}(x_0))}\leq CT^{-\frac{1}{2}}e^{CR^{1-N}\left(1+\frac{1}{T-t_0}\right)T}
	F(R).
\end{eqnarray}
Thus it suffices to prove \eqref{yu-6-22-16} in the case that $\mu\neq0$.
To this end,  define for each $\mu\in\mathbb{R}\setminus\{0\}$,
\begin{equation*}\label{yu-6-19-1}
	p(x,\xi,\mu)=e^{i\sqrt{|\mu|}\xi}\hat{v}(x,\mu)\quad\text{for a.e.}\quad (x,\xi)\in\triangle_R(x_0)\times\mathbb R.
\end{equation*}
Then,  $p(\cdot,\cdot,\mu)$ verifies 
\begin{equation*}\label{yu-6-19-2}
	\mbox{div}(A\nabla p(\cdot,\cdot,\mu))+i\mbox{sign}(\mu)l\partial_{\xi\xi}p(\cdot,\cdot,\mu)
	-bp(\cdot,\cdot,\mu)=0\;\;\mbox{in}\;\;\triangle_R(x_0)\times\mathbb{R}.
\end{equation*}
Here 
\begin{equation*}\label{yu-6-19-3}
     \mbox{sign}(\mu):=
\begin{cases}
	1&\mbox{if}\;\;\mu>0,\\
	-1&\mbox{if}\;\;\mu<0.
\end{cases}
\end{equation*}

	Let $m\in\mathbb{N}^+$ and $a_j=1-\frac{j}{2m}$ for $j=0,1,\dots,m+1$. For each $j\in\{0,1,\cdots, m\}$, we define a cutoff function 
\begin{equation*}\label{yu-6-19-4}
	h_j(s):=
\begin{cases}
	0&\mbox{if}\;\;|s|>a_j,\\
	\frac{1}{2}\left[1+\cos\left(\frac{\pi(a_{j+1}-s)}{a_{j+1}-a_j}\right)\right]
	&\mbox{if}\;\;a_{j+1}\leq |s|\leq a_j,\\
	1&\mbox{if}\;\;|s|<a_{j+1}.
\end{cases}
\end{equation*}
Clearly,
\begin{equation*}\label{yu-6-20-2}
	|h_j'(s)|\leq m\pi \;\;\mbox{for any}\;\;s\in\mathbb{R}.
\end{equation*}
Denote
$p_j=\frac{\partial^jp}{\partial\xi^j}$, $j=0,1,\dots,m$. Then $p_j$ verifies 
\begin{equation}\label{yu-6-19-6}
	\mbox{div}(A\nabla p_j(\cdot,\cdot,\mu))+i\mbox{sign}(\mu)lp_{j+2}(\cdot,\cdot,\mu)
	-bp_j(\cdot,\cdot,\mu)=0\;\;\mbox{in}\;\;\triangle_R(x_0)\times\mathbb{R}.
\end{equation}	
	Let 
\begin{equation*}\label{yu-6-19-7}
	\eta_j(x,\xi)=h_j\left(\frac{|x-x_0|}{R}\right)h_j\left(\frac{\xi}{R}\right)\quad\text{for}\;\;(x,\xi)\in\triangle_R(x_0)\times\mathbb{R}.
\end{equation*}
	Multiplying first (\ref{yu-6-19-6}) by 
	$\bar{p}_j\eta_j^2$ and then integrating by parts over $D_j=\triangle_{a_jR}(x_0)\times(-a_jR,a_jR)$, we obtain
\begin{eqnarray}\label{yu-6-19-8}
	&\;&-\int_{D_j}\nabla\bar{p}_j\cdot(A\nabla p_j)\eta_j^2dxd\xi
	-i\mbox{sign}(\mu)\int_{D_j}l|p_{j+1}|^2\eta_j^2dxd\xi\nonumber\\
	&=&\int_{D_j}b|p_j|^2\eta_j^2dxd\xi
	+\int_{D_j}\nabla\eta_j^2\cdot(A\nabla p_i)\bar{p}_jdxd\xi+i\mbox{sign}(\mu)\int_{D_j}lp_{i+1}\partial_\xi\eta^2_j\bar{p}_jdxd\xi. 
\end{eqnarray}
Since $\nabla\bar{p}_j\cdot(A\nabla p_j)$ and $|p_{j+1}|^2$ are real-valued,
 by (\ref{yu-6-19-8}) we get 
\begin{eqnarray}\label{yu-6-19-9}
	&\;&\left(\int_{D_j}\nabla\bar{p}_j\cdot(A\nabla p_j)\eta_j^2dxd\xi\right)^2
	+|\mbox{sign}(\mu)|^2\left(\int_{D_j}l|p_{j+1}|^2\eta_j^2dxd\xi\right)^2\nonumber\\
	&\leq&3\left(\int_{D_j}|b||p_j|^2\eta_j^2dxd\xi\right)^2
	+3\left(\int_{D_j}|\nabla \eta_j^2\cdot(A\nabla p_j)||p_j|dxd\xi\right)^2\nonumber\\
	&\;&+3\left(\int_{D_j}l|p_{j+1}||p_j|\partial_{\xi}\eta^2_jdxd\xi\right)^2
	:=3\sum_{i=1}^3I_i.
\end{eqnarray}
\par
	Next, we will estimate $I_i$ $(i=1,2,3)$ one by one.
	For the term $I_1$, we  shall prove that 
\begin{eqnarray}\label{yu-10-15-1}
	I_1\leq \frac{\Lambda_1^{-2}}{16}\left(\int_{D_j}|\nabla p_j|^2\eta_j^2dxd\xi\right)^2
	+\frac{C(1+m^4)}{R^4}\int_{D_j}|p_j|^2dxd\xi.
\end{eqnarray}
	We divide its proof into two cases. 
\vskip 5pt
         \emph{Case I.   $|b(x)|\leq \frac{\Lambda_2}{|x|}$ a.e. $x\in\Omega$.}	By the Hardy inequality, we derive that 
 \begin{eqnarray*}\label{yu-6-20-3}
	&\;&\int_{D_j}|b||p_j|^2\eta_j^2dxd\xi
	\leq\frac{4\epsilon_1}{(N-2)^2}
	\int_{D_j}(|\nabla p_j|^2\eta_j^2+|p_j|^2|\nabla\eta_j|^2)dxd\xi
	+\frac{\Lambda_2^2}{2\epsilon_1}
	\int_{D_j}|p_j|^2\eta_j^2dxd\xi\nonumber\\
	&\leq&\frac{4\epsilon_1}{(N-2)^2}
	\int_{D_j}|\nabla p_j|^2\eta_j^2dxd\xi
	+\frac{[8\pi^2m^2\epsilon_1^2+\Lambda_2^2R^2(N-2)^2]}{2\epsilon_1R^2(N-2)^2}
	\int_{D_j}|p_j|^2dxd\xi,\;\;\forall \epsilon_1>0.
\end{eqnarray*}
	Therefore, 
\begin{eqnarray*}\label{yu-6-20-4}
	I_1&\leq&\frac{2^5\epsilon_1^2}{(N-2)^4}
	\left(\int_{D_j}|\nabla p_j|^2\eta_j^2dxd\xi\right)^2
        +\frac{[8\pi^2m^2\epsilon_1^2+\Lambda_2^2R^2(N-2)^2]^2}{2\epsilon_1^2R^4(N-2)^4}\left(\int_{D_j}|p_j|^2dxd\xi\right)^2,
\end{eqnarray*}
         Let $\epsilon_1=\frac{(N-2)^2}{2^5\Lambda_1}$, we derive 
         (\ref{yu-10-15-1}).  
\vskip 5pt
\emph{Case II. $b(\cdot)\in L^{N+\delta}(\Omega)$ and $\|b(\cdot)\|_{L^{N+\delta}(\Omega)}\leq\Lambda_2$.}         
         By (\ref{yu-9-26-2}), (\ref{yu-9-29-1}) and the definition of $R_0$, we have       
 \begin{eqnarray*}\label{yu-10-15-2}
 	&\;&\int_{D_j}|b||p_j|^2\eta_j^2dxd\xi\nonumber\\
	&\leq& \Gamma_2\left(\triangle_{a_jR}(x_0),N,\frac{N}{2}\right)
	\|b\|_{L^N(\triangle_{a_jR}(x_0))}\int_{-a_jR}^{a_jR}\|p_j\eta_j\|_{L^2(\triangle_{a_jR}(x_0))}
	\|p_j\eta_j\|_{H^1(\triangle_{a_jR}(x_0))}d\xi\nonumber\\
	&\leq&\sqrt{2}\Theta_N^{\frac{N\delta}{N+\delta}}\Gamma_2\left(\triangle_1(0),N,\frac{N}{2}\right)
	R^{\frac{\delta}{N+\delta}}\|b\|_{L^{N+\delta}(\triangle_{a_jR}(x_0)}
	\int_{-a_jR}^{a_jR}\|\nabla(p_j\eta_j)\|_{L^2(\triangle_{a_jR}(x_0))}^2d\xi\nonumber\\
	&\leq&2\sqrt{2}\Theta_N^{\frac{N\delta}{N+\delta}}\Gamma_2\left(\triangle_1(0),N,\frac{N}{2}\right)
	\Lambda_2R_0^{\frac{\delta}{N+\delta}}\int_{D_j}\left(|\nabla p_j|^2\eta_j^2+\frac{m^2\pi^2}{R^2}|p_j|^2\right)dxd\xi\nonumber\\
	&\leq&\frac{\Lambda_1^{-1}}{4}\int_{D_j}|\nabla p_j|^2\eta_j^2dx
	+\frac{\Lambda_1^{-1}m^2\pi^2}{4R^2}\int_{D_j}|p_j|^2dx,
 \end{eqnarray*}
         which  gives (\ref{yu-10-15-1}).

          Moreover, 
\begin{eqnarray*}\label{yu-6-21-1}
	&\;&\int_{D_j}|\nabla \eta_j^2\cdot(A\nabla p_j)||p_j|dxd\xi
	\leq 2\Lambda_1\int_{D_j}|\nabla \eta_j||\eta_j||\nabla p_j||p_j|dxd\xi\nonumber\\
	&\leq&\epsilon_2\Lambda_1\int_{D_j}|\nabla p_j|^2\eta_j^2dxd\xi
	+\frac{\Lambda_1}{\epsilon_2}\int_{D_j}|p_j|^2|\nabla\eta_j|^2dxd\xi\nonumber\\
	&\leq&\epsilon_2\Lambda_1\int_{D_j}|\nabla p_j|^2\eta_j^2dxd\xi
	+\frac{\Lambda_1\pi^2m^2}{R^2\epsilon_2}\int_{D_j}|p_j|^2dxd\xi
\end{eqnarray*}
and
\begin{equation}\label{yu-6-21-2}
	I_2\leq 2\Lambda^2_1\epsilon_2^2\left(\int_{D_j}|\nabla p_j|^2\eta_j^2dxd\xi\right)^2
	+\frac{2\Lambda_1^2\pi^4m^4}{R^4\epsilon^2_2}\left(\int_{D_j}|p_j|^2dxd\xi\right)^2, \;\;\forall \epsilon_2>0.
\end{equation}
Further, 
\begin{eqnarray*}\label{yu-6-21-3}
	\int_{D_j}l|p_{j+1}|p_j|\partial_{\xi}\eta^2_jdxd\xi
	&\leq&\epsilon_3\int_{D_j}l|p_{j+1}|^2\eta_j^2dxd\xi
	+\frac{1}{\epsilon_3}\int_{D_j}l|p_j|^2|\partial_\xi\eta_j|^2dxd\xi\nonumber\\
	&\leq&\epsilon_3\int_{D_j}|p_{j+1}|^2\eta_j^2dxd\xi
	+\frac{m^2\pi^2\Lambda_3}{R^2\epsilon_3}\int_{D_j}|p_j|^2dxd\xi
\end{eqnarray*}	
and
\begin{equation}\label{yu-6-21-4}
	I_3\leq 2\epsilon_3^2\left(\int_{D_j}l|p_{j+1}|^2\eta_j^2dxd\xi\right)^2
	+\frac{2\pi^4\Lambda_3^2m^4}{R^4\epsilon_3^2}\left(\int_{D_j}|p_j|^2dxd\xi\right)^2, \;\;\forall \epsilon_3>0.
\end{equation}
Taking $\epsilon_2=\frac{\sqrt{2}}{4\Lambda_1^2}$, $\epsilon_3=\frac{1}{4}$
 in  (\ref{yu-6-21-2})
		and (\ref{yu-6-21-4}), respectively, by (\ref{yu-10-15-1}),
we derive that 
\begin{eqnarray}\label{yu-6-21-8}
	\sum_{i=1}^3I_i&\leq&\frac{\Lambda_1^{-2}}{8}\left(\int_{D_j}|\nabla p_j|^2\eta_j^2dxd\xi
	\right)^2+\frac{1}{8}\left(\int_{D_j}l|p_{j+1}|^2\eta_j^2dxd\xi\right)^2\nonumber\\
	&\;&+\frac{M_1
		+M_2m^4}{R^4}
	\left(\int_{D_j}|p_j|^2\eta_j^2dxd\xi\right)^2
\end{eqnarray}	
with two positive constants $M_1$ and $M_2$. 
	On the other hand, by the uniform ellipticity condition (\ref{yu-11-28-2}),  we find that
\begin{equation*}\label{yu-6-22-2}
	\left(\int_{D_j}\nabla \bar{p}_j\cdot(A\nabla p_j)\eta_j^2dxd\xi\right)^2
	\geq \Lambda_1^{-2}\left(\int_{D_j}|\nabla p_j|^2\eta_j^2dxd\xi\right)^2.
\end{equation*}
	This, together with (\ref{yu-7-29-3}), (\ref{yu-6-19-9}) and (\ref{yu-6-21-8}), gives that for each $j\in\{0,1,\cdots,m-1\}$, 
\begin{eqnarray*}\label{yu-6-22-3}
	\int_{D_{j+1}}|p_{j+1}|^2\eta^2_jdxd\xi&\leq& \frac{2\Lambda_3\sqrt{2(M_1+M_2m^4)}}{R^2}
	\int_{D_j}|p_j|^2\eta^2_jdxd\xi\nonumber\\
	&\leq&\frac{\Pi(1+m^2)}{R^2}\int_{D_j}|p_j|^2\eta_j^2dxd\xi,
\end{eqnarray*}
	where $\Pi=2\Lambda_3\sqrt{2(M_1+M_2)}$. 
	 Here, we used the definition of $D_j$. Iterating (\ref{yu-6-22-3}) for each 
	$j\in\{0,1,\cdots,m-1\}$, by the fact that $p_0=p=\hat{v}$ we obtain 	
\begin{equation}\label{yu-6-22-4}
	\int_{\triangle_{\frac{R}{2}}(x_0)\times(-\frac{R}{2},\frac{R}{2})}|p_m|^2dxd\xi
	\leq2R\left[\frac{\Pi(1+m^2)}{R^2}\right]^m\int_{\triangle_R(x_0)}|\hat{v}(x,\mu)|^2dx.
\end{equation}
By Lemma \ref{yu-lemma-6-10-1}, we get that  for each $\mu\in\mathbb{R}$,
\begin{eqnarray}\label{yu-6-22-5}
	\|\hat{v}(\cdot,\mu)\|_{L^2(\triangle_{R}(x_0))}
	&\leq&CT^{-\frac{1}{2}}e^{CR^{1-N}\left(1+\frac{1}{T-t_0}\right)T}
	F(R).
\end{eqnarray}
	Therefore, by (\ref{yu-6-22-4}) and (\ref{yu-6-22-5}), we get that for each $m\in\mathbb{N}^+$, 
\begin{equation}\label{yu-6-22-6}
	\int_{\triangle_{\frac{R}{2}}(x_0)\times(-\frac{R}{2},\frac{R}{2})}|p_m|^2dxd\xi
	\leq CT^{-1}\left[\frac{\Pi(1+m^2)}{R^2}\right]^mRe^{CR^{1-N}\left(1+\frac{1}{T-t_0}\right)T}
	F^2(R).
\end{equation}
\par
	For any $\varphi\in L^2(\triangle_{\frac{R}{2}}(x_0);\mathbb{C})$, we define 
\begin{equation*}\label{yu-6-22-7}
	P_{\mu}(\xi):=\int_{\triangle_{\frac{R}{2}}(x_0)}p(x,\xi,\mu)\bar{\varphi}(x)dx, \;\;\;\xi\in\left(-\frac{R}{2},\frac{R}{2}\right).
\end{equation*}
	It is well known that the following interpolation inequality holds (See a proof in Appendix)
\begin{equation}\label{yu-6-22-8}
	\|f\|_{L^\infty(I)}\leq C\left(|I|\|f'\|_{L^2(I)}^2+\frac{1}{|I|}\|f\|_{L^2(I)}^2\right)^{\frac{1}{2}}
	\;\;\mbox{for each}\;\;f\in H^1(I),
\end{equation}
	where $I$ is an bounded nonempty interval of $\mathbb{R}$ and $|I|$ is the length. Therefore, by 
	(\ref{yu-6-22-6}) we have that for any $\xi\in(-\frac{R}{2},\frac{R}{2})$ and $m\in\mathbb{N}^+$,
\begin{eqnarray}\label{yu-6-22-9}
	|P_{\mu}^{(m)}(\xi)|&\leq& C\left(R\int_{-\frac{R}{2}}^{\frac{R}{2}}|P_{\mu}^{(m+1)}(\xi)|^2d\xi
	+\frac{1}{R}\int_{-\frac{R}{2}}^{\frac{R}{2}}|P_{\mu}^{(m)}(\xi)|^2d\xi\right)^{\frac{1}{2}}\nonumber\\
	&\leq&C\left(R\int_{\triangle_{\frac{R}{2}}(x_0)\times(-\frac{R}{2},\frac{R}{2})}
	|p_{m+1}|^2dxd\xi+\frac{1}{R}\int_{\triangle_{\frac{R}{2}}(x_0)\times(-\frac{R}{2},\frac{R}{2})}|p_m|^2dxd\xi\right)
	^{\frac{1}{2}}\|\varphi\|_{L^2(\triangle_{\frac{R}{2}}(x_0))}\nonumber\\
	&\leq&CT^{-\frac{1}{2}}e^{CR^{1-N}\left(1+\frac{1}{T-t_0}\right)T}F(R)\frac{[2\Pi(m+1)]^{m+1}}{R^m}
	\|\varphi\|_{L^2(\triangle_{\frac{R}{2}}(x_0))}.
\end{eqnarray}
	This implies that  $P_{\mu}(\cdot)$ can be 
	analytically extended to the complex plane (still denoted by the same notation)
\begin{equation*}\label{yu-6-22-10}
	E_0:=\left\{\xi\in\mathbb{C}:\mbox{Re}\,\xi\in\left(-\frac{R}{2},\frac{R}{2}\right)\;\;\mbox{and}\;\;\mbox{Im}\,\xi\in(-L_0,L_0)\right\},
\end{equation*}
	where $L_0:=\frac{R}{2e\Pi}$. 
Then, 
\begin{equation*}\label{yu-6-22-11}
	|P_{\mu}(\xi)|\leq \sum_{m=0}^\infty\frac{|P^{(m)}(0)|}{m!}|\xi|^m,
\end{equation*}
when $\xi\in i\mathbb{R}\cap E_0$.
Taking $\xi_0=-\frac{iR}{4e\Pi}$, by (\ref{yu-6-22-9}), we  get that 
\begin{equation}\label{yu-6-22-12}
	|P_{\mu}(\xi_0)|\leq C\Pi\sum_{m=0}^\infty\frac{(m+1)^{m+1}}{m!(2e)^m}
	T^{-\frac{1}{2}}
	e^{CR^{1-N}\left(1+\frac{1}{T-t_0}\right)T}F(R)
	\|\varphi\|_{L^2(\triangle_{\frac{R}{2}}(x_0))}.
\end{equation}
While, by the definition,
\begin{eqnarray*}\label{yu-6-22-13}
	P_{\mu}(\xi_0)=e^{\frac{\sqrt{|\mu|}R}{4e\Pi}}\int_{\triangle_{\frac{R}{2}}(x_0)}\hat{v}(x,\mu)\bar{\varphi}(x)dx.
\end{eqnarray*}
	This, together with (\ref{yu-6-22-12}), means that, 
	\begin{equation}\label{yu-6-22-14}
	\|\hat{v}(\cdot,\mu)\|_{L^2(\triangle_{\frac{R}{2}}(x_0))}\leq CT^{-\frac{1}{2}}
	e^{CR^{1-N}\left(1+\frac{1}{T-t_0}\right)T-\frac{\sqrt{|\mu|}R}{4e\Pi}}F(R).
\end{equation}
		By (\ref{yu-6-22-14}) and (\ref{yu-6-22-15}), we derive (\ref{yu-6-22-16}) and complete the proof.	
\end{proof}

\subsubsection{Stability estimate and three-ball inequality for elliptic equations}\label{yu-section-7-26-3}
Suppose $T>0$, $L>0$ and $\triangle_R(x_0)\subset\Omega$ with $x_0\in\Omega$.
Let $g\in H^1(\triangle_R(x_0)\times(-L,L))$ be a solution of the following elliptic equation	
\begin{equation}\label{yu-6-23-9}
\begin{cases}
	\mbox{div}(A(x)\nabla g)+l(x) g_{x_{N+1}x_{N+1}}-b(x)g=0&\;\;\;\text{in}\;\;\; \triangle_R(x_0)\times(-L,L),\\
	g(x,0)=f_1(x)&\;\;\;\text{in}\;\;\; \triangle_R(x_0),\\
	g_{x_{N+1}}(x,0)=f_2(x)&\;\;\;\text{in}\;\;\; \triangle_R(x_0),
\end{cases}
\end{equation}
	where $f_1\in H^1(\triangle_R(x_0))$, $f_2\in L^2(\triangle_R(x_0))$, 
	$A$, $b$ and $l$ satisfy the same assumptions as before.

\begin{lemma}[Stability estimate]\label{yu-proposition-7-1-1}
There is $\gamma\in(0,1)$ such that for any $r\in(0,\min\{R,L\}/3)$, 
{\begin{equation}\label{yu-7-3-b-2}
	\|g\|^2_{H^1(B_r(x_0,0))}\leq Cr^{-4}\|g\|_{H^1(B_{2r}(x_0,0))}^{2\gamma}\left(\|f_1\|^2_{H^1(\triangle_{2r}(x_0))}+\|f_2\|^2_{L^2(\triangle_{2r}(x_0))}\right)^{1-\gamma}.
\end{equation}}
\end{lemma}	

\medskip

The proof of Lemma \ref{yu-proposition-7-1-1} is based on a point-wise estimate (see Lemma \ref{yu-lemma-7-1-1} below).
Here and in the sequel, for simplicity we denote
\begin{equation*}\label{yu-7-1-6}
	\bar{A}(x,x_{N+1})=\left[\bar{a}^{ij}(x,x_{N+1})\right]_{(N+1)\times (N+1)}:=\mbox{diag}(A(x),l(x)),
\end{equation*}
$$
\nabla=(\nabla_x,\partial_{x_{N+1}}),\quad \mbox{div}=\mbox{div}_x+\partial_{x_{N+1}}
$$
when they do not arise any confusion in the context.	
	
\begin{lemma}\label{yu-lemma-7-1-1}
Let $s>0$, $\lambda>0$, $\varphi\in C^2(\overline{B_R}(x_0,0))$ and set $\alpha=e^{\lambda\varphi}$, $\theta=e^{s\alpha}$.
If $V\in C^2(\triangle_R(x_0)\times(-L,L))$ and $W=\theta V$, then the following inequality holds:
\begin{eqnarray*}\label{yu-7-2-1}
	&\;&\theta^2|\mbox{div}(\bar{A}\nabla V)|^2+\mathcal{D}\nonumber\\
	&\geq&\mathcal{B}_1|W|^2+\mathcal{B}_2\nabla W\cdot(\bar{A}\nabla W)
	+2s\lambda^2W\nabla[\alpha\nabla\varphi\cdot(\bar{A}\nabla\varphi)]\cdot(\bar{A}\nabla W)+2s\lambda^2\alpha|\nabla W\cdot (\bar{A}\nabla\varphi)|^2\nonumber\\
	&\;&+2s\lambda\alpha(\bar{A}\nabla W)\cdot[D^2\varphi(\bar{A}\nabla W)]+2s\lambda\alpha\left(\sum_{i,j=1}^{N+1}\partial_iW\nabla \bar{a}^{ij}\partial_j\varphi\right)\cdot(\bar{A}\nabla W)\nonumber\\
	&\;&-s\lambda\alpha\left(\sum_{i,j=1}^{N+1}\partial_i W\nabla \bar{a}^{ij}\partial_j W\right)\cdot(\bar{A}\nabla\varphi),
\end{eqnarray*}
	where
\begin{equation*}
\begin{cases}
	\mathcal{B}_1=s^3\lambda^4\alpha^3|\nabla\varphi\cdot(\bar{A}\nabla\varphi)|^2+s^3\lambda^3\alpha^3\mbox{div}\{\bar{A}\nabla\varphi[\nabla\varphi\cdot(\bar{A}\nabla\varphi)]\}\\
	\;\;\;\;\;\;\;\;\;-2s^2\lambda^2\alpha^2|\mbox{div}(\bar{A}\nabla\varphi)|^2-2s^2\lambda^4\alpha^2|\nabla\varphi\cdot(\bar{A}\nabla\varphi)|^2\\
	\;\;\;\;\;=s^3\lambda^4\alpha^3|\nabla\varphi\cdot(\bar{A}\nabla\varphi)|^2+{s^3\alpha^3O(\lambda^3)}
	+s^2\alpha^2O(\lambda^4),\\
	\mathcal{B}_2=s\lambda^2\alpha\nabla\varphi\cdot(\bar{A}\nabla\varphi)-s\lambda\alpha\mbox{div}(\bar{A}\nabla\varphi)\\
	\;\;\;\;\;=s\lambda^2\alpha\nabla\varphi\cdot(\bar{A}\nabla\varphi)+s\alpha O(\lambda),\\
	\mathcal{D}=2s\lambda^2\mbox{div}[\alpha W\bar{A}\nabla W\nabla\varphi\cdot(\bar{A}\nabla \varphi)]
	+2s\lambda\mbox{div}[\alpha \bar{A}\nabla W\nabla W\cdot(\bar{A}\nabla \varphi)]\\
	\;\;\;\;\;\;\;\;\;-s\lambda\mbox{div}[\alpha\nabla W\cdot(\bar{A}\nabla W)\bar{A}\nabla\varphi]
	+s^3\lambda^3\mbox{div}[\alpha^3|W|^2\bar{A}\nabla\varphi\nabla\varphi\cdot(\bar{A}\nabla\varphi)].
\end{cases}
\end{equation*}
\end{lemma}

	\smallskip
	
\begin{proof}[\textbf{Proof of Lemma \ref{yu-proposition-7-1-1}}] 
	
With the same notation as above,  (\ref{yu-6-23-9}) can be rewritten as
\begin{equation*}\label{yu-7-1-7}
	\mbox{div}(\bar{A}\nabla g)-bg=0\;\;\mbox{in}\;\;\triangle_R(x_0)\times(-L,L),
\end{equation*}
where $\bar A$ satisfies 
\begin{equation*}\label{yu-7-1-12}
	\Lambda_4^{-1}|\xi|^2\leq \bar{A}(x,x_{N+1})\xi\cdot\xi\leq \Lambda_4|\xi|^2\;\;
	\mbox{for each}\;\;(x,\xi)\in(\triangle_{R}(x_0)\times(-L,L))\times\mathbb{R}^{N+1},
\end{equation*}
with $\Lambda_4=\max\{\Lambda_1,\Lambda_3\}$.

We next divide the proof into two steps as follows.

\vskip 5pt
         \textbf{Step 1.} For each $r<\min\{R,L\}$, let us  set 
         $$r_1=r,\quad r_2=\frac{3r}{2},\quad r_3=2r,\quad r_4=3r$$
  and        
 \begin{equation*}\label{yu-7-2-10}
 	\omega_1=B^+_{r_1}(x_0,0),\;\;\omega_2=B^+_{r_2}(x_0,0),\;\;\omega_3=B^+_{r_3}(x_0,0),\;\;\omega_4=\triangle_{r_4}(x_0)\times(0,3r).
 \end{equation*}
	Let {$\varphi\in C^2(\mathbb{R};[0,4])$} be such that
\begin{equation}\label{yu-7-2-11}
\begin{cases}
	3<\varphi<4&\mbox{in}\;\;\omega_1,\\
	0<\varphi<1&\mbox{in}\;\;\omega_4\backslash\omega_2,\\
	|\nabla\varphi|>0&\mbox{in}\;\;\overline{\omega_4}.
\end{cases}
\end{equation}
Take a cutoff function $\eta\in C^\infty(\mathbb{R}^{N+1};[0,1])$ to be such that
\begin{equation*}\label{yu-7-2-12}
\begin{cases}
	\eta=1&\mbox{in}\;\;\overline{\omega_2},\\
	\eta=0&\mbox{in}\;\;\overline{\omega_4}\backslash\omega_3,\\
	|\mbox{div}(\bar{A}\nabla\eta)|+|\nabla\eta|^2\leq \frac{C}{r^2}&\mbox{in}\;\;\mathbb{R}^{N+1},
\end{cases}
\end{equation*}
	where $C$ is a generic constant independent of $r$.
		Setting $V=\eta g$,  we have
\begin{equation}\label{yu-7-2-13}
\begin{cases}
	\mbox{div}(\bar{A}\nabla V)-bV=\mbox{div}(\bar{A}\nabla\eta)g+2\nabla\eta\cdot(\bar{A}\nabla g)&\mbox{in}\;\;\omega_4,\\
	|\nabla V|=V=0&\mbox{on}\;\;\partial\omega_4\backslash(\triangle_{r_3}(x_0)\times\{0\}).
\end{cases}
\end{equation}
It follows from Lemma \ref{yu-lemma-7-1-1} that
\begin{eqnarray}\label{yu-7-2-14}
	&&\int_{\omega_4}\theta^2|\mbox{div}(\bar{A}\nabla V)|^2dxdx_{N+1}
	+\int_{\omega_4}\mathcal{D}dxdx_{N+1}\nonumber\\
	&\geq&\int_{\omega_4}\mathcal{B}_1|W|^2dxdx_{N+1}+\int_{\omega_3}\mathcal{B}_2\nabla W\cdot(\bar{A}\nabla W)dxdx_{N+1}\nonumber\\
	&\;&+2s\lambda^2\int_{\omega_4} W\nabla[\alpha\nabla\varphi\cdot(\bar{A}\nabla\varphi)]\cdot (\bar{A}\nabla W)dxdx_{N+1}\nonumber\\
	&\;&+2s\lambda\int_{\omega_4}\alpha(\bar{A}\nabla W)\cdot[D^2\varphi(\bar{A}\nabla W)]dxdx_{N+1}
	+2s\lambda^2\int_{\omega_4}\alpha|\nabla W\cdot(\bar{A}\nabla\varphi)|^2dxdx_{N+1}\nonumber\\
	&\;&+2s\lambda\int_{\omega_4}\alpha\left(\sum_{i,j=1}^{N+1}\partial_iW\nabla \bar{A}^{ij}\partial_j\varphi\right)\cdot(\bar{A}\nabla W)dxdx_{N+1}\nonumber\\
	&\;&-s\lambda\int_{\omega_4}\alpha\left(\sum_{i,j=1}^{N+1}\partial_iW\nabla \bar{A}^{ij}\partial_jW\right)\cdot(\bar{A}\nabla\varphi)dxdx_{N+1}.
\end{eqnarray}
 By the Cauchy-Schwarz inequality, we find
\begin{equation}\label{yu-7-2-15}
	2s\lambda^2\left| W\nabla[\alpha\nabla\varphi\cdot(\bar{A}\nabla\varphi)]\cdot (\bar{A}\nabla W)\right|\leq C\lambda^2(s^2\lambda^2\alpha|W|^2+\alpha|\nabla W|^2),
\end{equation}
\begin{equation}\label{yu-7-2-16}
	2s\lambda\alpha\left|(\bar{A}\nabla W)\cdot[D^2\varphi(\bar{A}\nabla W)]\right|\leq Cs\lambda\alpha|\nabla W|^2,
\end{equation}
\begin{equation}
	2s\lambda\alpha\left|\left(\sum_{i,j=1}^{N+1}\partial_iW\nabla \bar{a}^{ij}\partial_j\varphi\right)\cdot(\bar{A}\nabla W)\right|\leq Cs\lambda\alpha|\nabla W|^2
\end{equation}
	and 
\begin{equation}\label{yu-7-2-17}
	s\lambda\alpha\left|\left(\sum_{i,j=1}^{N+1}\partial_iW\nabla \bar{a}^{ij}\partial_jW\right)\cdot(\bar{A}\nabla\varphi)\right|
	\leq Cs\lambda\alpha|\nabla W|^2.
\end{equation}
	By definitions of $\mathcal{B}_1$ and $\mathcal{B}_2$, we get 
\begin{equation}\label{yu-7-2-18}
	\mathcal{B}_1|W|^2\geq (s^3\lambda^4\alpha^3\Lambda^{-2}_1|\nabla\varphi|^2+s^3\alpha^3O(\lambda^3)+s^2\alpha^2O(\lambda^4))|W|^2
\end{equation}
	and 
\begin{equation}\label{yu-7-2-19}
	\mathcal{B}_2|\nabla W|^2\geq {(s\lambda^2\alpha\Lambda^{-1}_1|\nabla\varphi|^2+s\alpha O(\lambda))|\nabla W|^2}.
\end{equation}
	From (\ref{yu-7-2-14})--(\ref{yu-7-2-19}) and the positivity of $|\nabla\varphi|$, we have 
\begin{eqnarray*}\label{yu-7-2-20}
	&\;&\int_{\omega_4}\theta^2|\mbox{div}(\bar{A}\nabla V)|^2dxdx_{N+1}
	+\int_{\omega_4}\mathcal{D}dxdx_{N+1}\nonumber\\
	&\geq&C\int_{\omega_4}(s^3\lambda^4\alpha^3+s^3\alpha^3O(\lambda^3)+s^2\alpha^2 O(\lambda^4)
	-{Cs^2\lambda^4\alpha})|W|^2dxdx_{N+1}\nonumber\\
	&\;&+C\int_{\omega_4}[s\lambda^2\alpha+s\alpha O(\lambda)-C(\lambda^2+s\lambda)\alpha]|\nabla W|^2dxdx_{N+1}.
\end{eqnarray*}
	Therefore, there is a constant $\lambda_0>1$ such that for any $\lambda\geq \lambda_0$, one can find $s_0>1$ such that 
	for any $s\geq s_0$, 
\begin{eqnarray}\label{yu-7-2-21}
	\int_{\omega_4}\theta^2|\mbox{div}(\bar{A}\nabla V)|^2dxdx_{N+1}	
	+\int_{\omega_4}\mathcal{D}dxdx_{N+1}\nonumber\\
	\geq Cs^3\lambda^4\int_{\omega_4}\alpha^3|W|^2dxdx_{N+1}
	+Cs\lambda^2\int_{\omega_4}\alpha|\nabla W|^2dxdx_{N+1}.
\end{eqnarray}
By the definition of $\mathcal{D}$, we obtain 
\begin{eqnarray}\label{yu-7-2-22}
\int_{\omega_4}\mathcal{D}dxdx_{N+1} &=&2s\lambda^2\int_{\triangle_{r_4}(x_0)\times\{0\}}\alpha W (\bar{A}\nabla W)\cdot\vec{n}\nabla\varphi\cdot(\bar{A}\nabla\varphi)d\Gamma\nonumber\\
	&\;&+2s\lambda\int_{\triangle_{r_4}(x_0)\times\{0\}}\alpha(\bar{A}\nabla W)\cdot \vec{n}\nabla W\cdot(\bar{A}\nabla\varphi)d\Gamma\nonumber\\
	&\;&-s\lambda\int_{\triangle_{r_4}(x_0)\times\{0\}}\alpha\nabla W\cdot(\bar{A}\nabla W)(\bar{A}\nabla\varphi)\cdot\vec{n}d\Gamma\nonumber\\
	&\;&+s^3\lambda^3\int_{\triangle_{r_4}(x_0)\times\{0\}}\alpha^3|W|^2(\bar{A}\nabla\varphi)\cdot\vec{n}\nabla\varphi\cdot(\bar{A}\nabla\varphi)d\Gamma\nonumber\\
	&\leq& Cs\lambda\int_{\triangle_{r_4}(x_0)\times\{0\}}\alpha |\nabla W|^2d\Gamma
	+Cs^3\lambda^3\int_{\triangle_{r_4}(x_0)\times\{0\}}\alpha^3|W|^2d\Gamma.
\end{eqnarray}
From (\ref{yu-7-2-21}) and (\ref{yu-7-2-22}), we have 
\begin{eqnarray}\label{yu-7-3-1}
	&\;&Cs^3\lambda^4\int_{\omega_4}\alpha^3|W|^2dxdx_{N+1}+Cs\lambda^2\int_{\omega_4}\alpha|\nabla W|^2dxdx_{N+1}\nonumber\\
	&\leq&\int_{\omega_4}\theta^2|\mbox{div}(\bar{A}\nabla V)|^2dxdx_{N+1}+Cs\lambda\int_{\triangle_{r_4}(x_0)\times\{0\}}\alpha
	|\nabla W|^2d\Gamma \nonumber\\
&&\;\;+Cs^3\lambda^3\int_{\triangle_{r_4}(x_0)\times\{0\}}\alpha^3|W|^2d\Gamma.
\end{eqnarray}
        \textbf{Step 2.} Now, we return $W$ in (\ref{yu-7-3-1}) to $V$.  Note that 
\begin{eqnarray}\label{yu-7-3-2}
	\frac{1}{C}\theta^2(|\nabla V|^2+s^2\lambda^2\alpha^2|V|^2)\leq
	|\nabla W|^2+s^2\lambda^2\alpha^2|W|^2\leq
 C\theta^2(|\nabla V|^2+s^2\lambda^2\alpha^2|V|^2).
\end{eqnarray}
	 Based on the case of  the potential  $b$, by Lemma \ref{hardy}, (\ref{yu-9-26-2}) with $\epsilon=0$, the first equation in (\ref{yu-7-2-13}) and the fact that $\omega_4$ is a rectangle domain, we have
\begin{eqnarray}\label{yu-7-3-b-1}
	&\;&\int_{\omega_4}\theta^2|\mbox{div}(\bar{A}\nabla V)|^2dxdx_{N+1}\nonumber\\
	&=&\int_{\omega_4}\theta^2|\mbox{div}(\bar{A}\nabla\eta)g+2\nabla\eta\cdot(\bar{A}\nabla g)+bV|^2dxdx_{N+1}\nonumber\\
	&\leq&C\int_{\omega_4}\theta^2|\mbox{div}(\bar{A}\nabla\eta)g+2\nabla\eta\cdot(\bar{A}\nabla g)|^2dxdx_{N+1}
	+C\int_{\omega_4}\theta^2|bV|^2dxdx_{N+1}\nonumber\\
	&=&C\int_{\omega_4}\theta^2|\mbox{div}(\bar{A}\nabla\eta)g+2\nabla\eta\cdot(\bar{A}\nabla g)|^2dxdx_{N+1}	+C\int_{\triangle_{r_4}(x_0)\times(0,r_4)}|bW|^2dxdx_{N+1}\nonumber\\
	&\leq&C\int_{\omega_4}\theta^2|\mbox{div}(\bar{A}\nabla\eta)g+2\nabla\eta\cdot(\bar{A}\nabla g)|^2dxdx_{N+1}	+C\int_{\omega_4}(|\nabla W|^2+|W|^2)dxdx_{N+1}.
\end{eqnarray}
	Therefore, by (\ref{yu-7-3-1})--(\ref{yu-7-3-b-1}) and taking $\lambda_0>1$ large enough, we get 
{\begin{eqnarray}\label{yu-7-3-3}
	&\;&Cs^3\lambda^4\int_{\omega_4}\alpha^3\theta^2|V|^2dxdx_{N+1}+Cs\lambda^2\int_{\omega_4}\alpha\theta^2|\nabla V|^2dxdx_{N+1}\nonumber\\
	&\leq&C\int_{\omega_4}\theta^2|\mbox{div}(\bar{A}\nabla\eta)g+2\nabla\eta\cdot(\bar{A}\nabla g)|^2dxdx_{N+1}\nonumber\\
	&\;&+Cs\lambda\int_{\triangle_{r_4}(x_0)\times\{0\}}\alpha\theta^2|\nabla V|^2d\Gamma
	+Cs^3\lambda^3\int_{\triangle_{r_4}(x_0)\times\{0\}}\alpha^3\theta^2|V|^2d\Gamma.
\end{eqnarray}}
	By the definition of $\varphi$ (see (\ref{yu-7-2-11})), we know that 
\begin{equation*}
\begin{cases}
	\alpha\geq e^{3\lambda}\;\;\mbox{and}\;\;\theta\geq e^{se^{3\lambda}}&\mbox{in}\;\;\omega_1,\\
	\alpha\leq e^{\lambda}\;\;\mbox{and}\;\;\theta\leq e^{se^{\lambda}}&\mbox{in}\;\;\overline{\omega_4}\backslash\omega_2. 
\end{cases}
\end{equation*}
	Moreover, by the definition of $\eta$, we have 
\begin{equation*}\label{yu-7-3-4}
	\nabla \eta=0\;\;\mbox{in}\;\;\omega_2\cup (\overline{\omega_4}\backslash\omega_3).
\end{equation*}
	By the fact that $V=\eta g$, one can get  
\begin{eqnarray}\label{yu-7-3-5}
	&\;&Cs^3\lambda^4\int_{\omega_4}\alpha^3\theta^2|V|^2dxdx_{N+1}+Cs\lambda^2\int_{\omega_4}\alpha \theta^2|\nabla V|^2dxdx_{N+1}\nonumber\\
	&\geq&Cs^3\lambda^4\int_{\omega_1}\alpha^3\theta^2|g|^2dxdx_{N+1}+Cs\lambda^2\int_{\omega_1}\alpha \theta^2|\nabla g|^2dxdx_{N+1}\nonumber\\
	&\geq&Cs^3\lambda^4e^{9\lambda}e^{2se^{3\lambda}}\int_{\omega_1}|g|^2dxdx_{N+1}
	+Cs\lambda^2e^{3\lambda}e^{2se^{3\lambda}}\int_{\omega_1}|\nabla g|^2dxdx_{N+1}.
\end{eqnarray}
	Moreover, 
\begin{eqnarray}\label{yu-7-3-6}
	&\;&\int_{\omega_4}\theta^2|\mbox{div}(\bar{A}\nabla\eta)g+2\nabla\eta\cdot(\bar{A}\nabla g)\eta|^2dxdx_{N+1}\nonumber\\
	&\leq&\frac{C}{r^4}\int_{\omega_3\backslash\omega_2}\theta^2(|g|^2+|\nabla g|^2)dxdx_{N+1}	\leq\frac{C}{r^4}e^{2se^\lambda}\int_{\omega_3\backslash\omega_2}(|g|^2+|\nabla g|^2)dxdx_{N+1}.
\end{eqnarray}
	Further, 
{\begin{eqnarray}\label{yu-7-3-7}
	&\;&Cs\lambda\int_{\triangle_{r_4}(x_0)\times\{0\}}\alpha \theta^2|\nabla V|^2d\Gamma+Cs^3\lambda^3\int_{\triangle_{r_4}(x_0)\times\{0\}}\alpha^3\theta^2|V|^2d\Gamma\nonumber\\
	&\leq&\frac{Cs^3\lambda^3}{r^2}e^{3\lambda}e^{2se^{4\lambda}}
	\int_{\triangle_{r_3}(x_0)\times\{0\}}|g|^2d\Gamma+Cs\lambda e^\lambda e^{2se^{4\lambda}}\int_{\triangle_{r_3}(x_0)\times\{0\}}
	|\nabla g|^2d\Gamma.
\end{eqnarray}}
	Combining (\ref{yu-7-3-3})--(\ref{yu-7-3-7}), we have 
\begin{eqnarray*}\label{yu-7-3-8}
	Ce^{3\lambda}e^{2se^{3\lambda}}\int_{\omega_1}(|g|^2+|\nabla g|^2)dxdx_{N+1}	&\leq&\frac{1}{r^4}e^{2se^{\lambda}}\int_{\omega_3}(|g|^2+|\nabla g|^2)dxdx_{N+1}\nonumber\\
	&\;&+\frac{s^3\lambda^3}{r^2}e^{3\lambda}e^{2se^{4\lambda}}\int_{\triangle_{r_3}(x_0)\times\{0\}}(|g|^2+
	|\nabla g|^2)d\Gamma.
\end{eqnarray*}
Hence,
{\begin{eqnarray}\label{yu-7-3-9}
	C\int_{\omega_1}(|g|^2+|\nabla g|^2)dxdx_{N+1}&\leq&\frac{1}{r^4}e^{-3\lambda}e^{2s(e^\lambda-e^{3\lambda})}
	\int_{\omega_3}(|g|^2+|\nabla g|^2)dxdx_{N+1}\nonumber\\
	&\;&+\frac{s^3\lambda^3}{r^2}e^{2s(e^{4\lambda}-e^{3\lambda})}\int_{\triangle_{r_3}(x_0)\times\{0\}}(|g|^2+|\nabla g|^2)d\Gamma.
\end{eqnarray}}
	Fix $\lambda:=\lambda_0 >1$ and define
\begin{equation*}\label{yu-7-3-10}
	\epsilon:=e^{-3\lambda_0}e^{2s(e^{\lambda_0}-e^{3\lambda_0})},\;\;
	\mu:=\frac{2s(e^{4\lambda_0}-e^{3\lambda_0})+3(\ln s+\ln \lambda_0)}{2s(e^{3\lambda_0}-e^{\lambda_0})
	+3\lambda_0},
\end{equation*}
\begin{equation*}\label{yu-7-3-11}
	\epsilon_0:=e^{-3\lambda_0}e^{2s_0(e^{\lambda_0}-e^{3\lambda_0})}.
\end{equation*}
	So, (\ref{yu-7-3-9}) can be rewritten by 
\begin{multline}\label{yu-7-3-12}
	C\int_{\omega_1}(|g|^2+|\nabla g|^2)dxdx_{N+1}\\
	\leq \frac{\epsilon}{r^4}\int_{\omega_3}(|g|^2+|\nabla g|^2)dxdx_{N+1}	+\frac{\epsilon^{-\mu}}{r^2}\int_{\triangle_{r_3}(x_0)\times\{0\}}(|g|^2+|\nabla g|^2)d\Gamma.
\end{multline}

	We treat two cases separately. 
	
	$\bullet$ If 
\begin{equation*}\label{yu-7-3-13}
	\left(\frac{\int_{\triangle_{r_3}(x_0)\times\{0\}}(|g|^2+|\nabla g|^2)d\Gamma}{\int_{\omega_3}(|g|^2+|\nabla g|^2)dxdx_{N+1}	}\right)^{\frac{1}{1+\mu}}> \epsilon_0. 
\end{equation*}
	Then 
{\begin{eqnarray}\label{yu-7-3-14}
	&\;&C\int_{\omega_1}(|g|^2+|\nabla g|^2)dxdx_{N+1}\nonumber\\
	&=&C\left(\int_{\omega_3}(|g|^2+|\nabla g|^2)dxdx_{N+1}\right)^{1-\frac{1}{1+\mu}}
	\left(\int_{\omega_3}(|g|^2+|\nabla g|^2)dxdx_{N+1}\right)^{\frac{1}{1+\mu}}\nonumber\\
	&\leq&C\left(\int_{\omega_3}(|g|^2+|\nabla g|^2)dxdx_{N+1}\right)^{1-\frac{1}{1+\mu}}
	\left(\int_{\triangle_{r_3}(x_0)\times\{0\}}(|g|^2+|\nabla g|^2)d\Gamma\right)^{\frac{1}{1+\mu}}.
\end{eqnarray}}

$\bullet$	If 
{\begin{equation*}\label{yu-7-3-15}
	\left(\frac{\int_{\triangle_{r_3}(x_0)\times\{0\}}(|g|^2+|\nabla g|^2)d\Gamma}{\int_{\omega_3}(|g|^2+|\nabla g|^2)dxdx_{N+1}	}\right)^{\frac{1}{1+\mu}}\leq\epsilon_0. 
\end{equation*}}
	In this case, we choose a $s>s_0$ such that  {$\epsilon=\left(\frac{\int_{\triangle_{r_3}(x_0)\times\{0\}}(|g|^2+|\nabla g|^2)d\Gamma}{\int_{\omega_3}(|g|^2+|\nabla g|^2)dxdx_{N+1}
	}\right)^{\frac{1}{1+\mu}}$} in (\ref{yu-7-3-12}), we have 	
{\begin{eqnarray}\label{yu-7-3-16}	
	&\;&C\int_{\omega_1}(|g|^2+|\nabla g|^2)dxdx_{N+1}\nonumber\\
	&\leq&\frac{1}{r^4}\left(\int_{\omega_3}(|g|^2+|\nabla g|^2)dxdx_{N+1}\right)^{1-\frac{1}{1+\mu}}
	\left(\int_{\triangle_{r_3}(x_0)\times\{0\}}(|g|^2+|\nabla g|^2)d\Gamma\right)^{\frac{1}{1+\mu}}.
\end{eqnarray}}
	Combining (\ref{yu-7-3-14}) and (\ref{yu-7-3-16}), we get that 
{\begin{equation}\label{yu-7-3-17}
	\|g\|^2_{H^1(B_{r_1}^+(x_0,0))}\leq Cr^{-4}\|g\|_{H^1(B_{r_3}^+(x_0,0))}^{\frac{2}{1+\mu}}
	\left(\|g(\cdot,0)\|^2_{H^1(\triangle_{r_3}(x_0))}+\|\partial_{{N+1}}g(\cdot,0)\|^2_{L^2(\triangle_{r_3}(x_0))}\right)^{\frac{\mu}{1+\mu}}.
\end{equation}}

Note that $g$ is an even function with respect to the variable ${x_{N+1}}$. So, by (\ref{yu-7-3-17}), we have (\ref{yu-7-3-b-2}) and the proof is completed. 
\end{proof}

\par

\begin{lemma}[Three-ball inequality]
\label{yu-lemma-7-4-2}
 There is $\beta\in(0,1)$ such that for any
	$r\in(0,\frac{1}{12}\min\{R,L\})$, the inequality 
	\begin{equation}\label{yu-7-31-1-bb}
	\|g\|_{L^2(B_{6r}(x_0,0))}\leq C(r)\|g\|^\beta_{L^2(B_{r}(x_0,0))}\|g\|_{L^2(B_{8r}(x_0,0))}^{1-\beta}
\end{equation}
holds for all solutions of 
\begin{equation*}\label{yu-7-26-1}
	\mbox{div}(A(x)\nabla g)+l(x)g_{x_{N+1}x_{N+1}}-b(x)g=0\;\;\mbox{in}\;\;\triangle_R(x_0)\times(-L,L).
\end{equation*}
\end{lemma}

\begin{proof}
	
We divide the proof into the following two steps.
\vskip 5pt
   \textbf{Step 1.} 
   For any $r<\min\{R,L\}$, let us set 
 $$r_1=r,\quad r_2=6r,\quad r_3=8r,\quad r_4=12r.$$ 
Take 
\begin{equation}\label{yu-7-31-1}
	\varphi(x,x_{N+1})=r_4^2-|x-x_0|^2-x_{N+1}^2, \quad (x,x_{N+1})\in B_{r_4}(x_0,0),
\end{equation}
	 and set a cutoff function $\eta\in C^\infty(\mathbb{R}^{N+1};[0,1])$ to be such that 
\begin{equation*}\label{yu-7-31-2}
\begin{cases}
	\eta=0&\mbox{in}\;\;\overline{B_{\frac{r_1}{2}}(x_0,0)},\\
	\eta=1&\mbox{in}\;\;\overline{B_{\frac{r_2+r_3}{2}}(x_0,0)}\backslash B_{\frac{3r_1}{4}}(x_0,0),\\
	\eta=0&\mbox{in}\;\;\overline{B_{r_4}(x_0,0)}\backslash B_{\frac{r_2+3r_3}{4}}(x_0,0)\\
	|\mbox{div}\bar{A}\nabla \eta|+|\nabla\eta|^2\leq \frac{C}{r^2}&\mbox{in}\;\;\mathbb{R}^{N+1} ,
\end{cases}
\end{equation*}
	where $C>0$ is a positive constant  independent of $r$.  Let $V=\eta g$. Then, 
\begin{equation*}\label{yu-7-31-2}
\begin{cases}
	\mbox{div}(\bar{A}\nabla V)-bV=\mbox{div}(\bar{A}\nabla\eta) g+2\nabla\eta\cdot(\bar{A}\nabla g)&\mbox{in}\;\;
	B_{r_4}(x_0,0),\\
	|\nabla V|=V=0&\mbox{on}\;\;\partial B_{r_4}(x_0,0). 
\end{cases}
\end{equation*}
Taking $W:=\theta V$ and repeating the proof of Step 1 in  Lemma \ref{yu-proposition-7-1-1}, one can claim that there is  $\lambda_0(r)>0$ such that for any $\lambda\geq \lambda_0(r)$, one can find $s_0(r)>1$ such that 
	$s\geq s_0$,
\begin{eqnarray*}\label{yu-7-31-3}
	C(r)s^3\lambda^4\int_{B_{r_4}(x_0,0)}\alpha^3|W|^2dxdx_{N+1}+C(r)s\lambda^2\int_{B_{r_4}(x_0,0)}\alpha|\nabla W|^2dxdx_{N+1}\nonumber\\
	\leq \int_{B_{r_4}(x_0,0)}\theta^2|\mbox{div}(\bar{A}\nabla V)|^2dxdx_{N+1}.
\end{eqnarray*}	
	Similar to the proof of (\ref{yu-7-3-3}), we can get  
\begin{eqnarray}\label{yu-7-31-4}
	C(r)s^3\lambda^4\int_{B_{r_4}(x_0,0)}\alpha^3\theta^2|V|^2dxdx_{N+1}
	+C(r)s\lambda^2\int_{B_{r_4}(x_0,0)}\alpha\theta^2|\nabla V|^2dxdx_{N+1}\nonumber\\
	\leq \int_{B_{r_4}(x_0,0)}\theta^2|\mbox{div}(\bar{A}\nabla \eta)g+2\nabla\eta\cdot(\bar{A}\nabla g)|^2dxdx_{N+1}.
\end{eqnarray}
\vskip 5pt
	\textbf{Step 2.} By the definition of $\varphi$ (see (\ref{yu-7-31-1})), we have 
\begin{equation*}\label{yu-7-31-5}
\begin{cases}
	\alpha\geq e^{108\lambda r^2}\geq 1,\;\;\theta\geq e^{se^{108\lambda r^2}}
	&\mbox{in}\;\;\overline{B_{r_2}(x_0,0)}\backslash B_{r_1}(x_0,0),\\
	\theta\leq e^{se^{144\lambda r^2}}&\mbox{in}\;\;\overline{B_{\frac{3r_1}{4}}(x_0,0)},\\
	\theta\leq e^{se^{95\lambda r^2}}&\mbox{in}\;\;
	\overline{B_{\frac{r_2+3r_3}{4}}(x_0,0)}\backslash B_{\frac{r_2+r_3}{2}}(x_0,0). 
\end{cases}
\end{equation*}
	Further, 
\begin{equation*}\label{yu-7-31-6}
	|\mbox{div}(\bar{A}\nabla\eta)|=|\nabla \eta|=0\;\;\mbox{in}\;\;\overline{B}_{\frac{r_1}{2}(x_0,0)}\bigcup \left(\overline{B_{\frac{r_2+r_3}{2}}(x_0,0)}\backslash B_{\frac{3r_1}{4}}(x_0,0)\right)\bigcup \left(\overline{B_{r_3}(x_0,0)}\backslash B_{\frac{r_2+3r_3}{4}}(x_0,0)\right).
\end{equation*}
	Hence, from the fact $V=\eta g$, we have 
\begin{equation}\label{yu-7-31-7}
	C(r)s^3\lambda^4\int_{B_{r_4}(x_0,0)}\alpha^3\theta^2|V|^2dxdx_{N+1}
	\geq C(r)s^3\lambda^4e^{2se^{108\lambda r^2}}\int_{B_{r_2}(x_0,0)\backslash B_{r_1}(x_0,0)}|g|^2dxdx_{N+1},
\end{equation}
	and 
\begin{eqnarray}\label{yu-7-31-8}
	&\;&\int_{B_{r_4}(x_0,0)}\theta^2|\mbox{div}(\bar{A}\nabla \eta)g+2\nabla\eta\cdot(\bar{A}\nabla g)|^2dxdx_{N+1}\nonumber\\
	&\leq& e^{2se^{144\lambda r^2}}\int_{B_{\frac{3r_1}{4}}(x_0,0)\backslash B_{\frac{r_1}{2}}(x_0,0)}
	\left(\frac{1}{r^4}|g|^2+\frac{1}{r^2}|\nabla g|^2\right)dxdx_{N+1}\nonumber\\
	&\;&+e^{2se^{95\lambda r^2}}\int_{B_{\frac{r_2+3r_3}{4}}(x_0,0)\backslash B_{\frac{r_2+r_3}{2}}(x_0,0)}
	\left(\frac{1}{r^4}|g|^2+\frac{1}{r^2}|\nabla g|^2\right)dxdx_{N+1}.
\end{eqnarray}
	By the interior estimate of elliptic equations
\begin{equation*}\label{yu-7-31-9}
	\int_{B_{\frac{3r_1}{4}}(x_0,0)\backslash B_{\frac{r_1}{2}}(x_0,0)}|\nabla g|^2dxdx_{N+1}
	\leq\frac{C}{r^2}\int_{B_{r_1}(x_0,0)}|g|^2dxdx_{N+1}
\end{equation*}	
	and 
\begin{equation*}\label{yu-7-31-10}
	\int_{B_{\frac{r_2+3r_3}{4}}(x_0,0)\backslash B_{\frac{r_2+r_3}{2}}(x_0,0)}|\nabla g|^2dxdx_{N+1}
	\leq \frac{C}{r^2}\int_{B_{r_3}(x_0,0)\backslash B_{r_2}(x_0,0)}|g|^2dxdx_{N+1}. 
\end{equation*}
	These, along with (\ref{yu-7-31-8}), yield that 
\begin{eqnarray}\label{yu-7-31-11}
	&\;&\int_{B_{r_4}(x_0,0)}\theta^2|\mbox{div}(\bar{A}\nabla \eta)g
	+2\nabla\eta\cdot(\bar{A}\nabla g)|^2dxdx_{N+1}\nonumber\\
	&\leq&C\frac{1}{r^4}e^{2se^{144\lambda r^2}}\int_{B_{r_1}(x_0,0)}|g|^2dxdx_{N+1}
	+C\frac{1}{r^4}e^{2se^{95\lambda r^2}}\int_{B_{r_3}(x_0,0)\backslash B_{r_2}(x_0,0)}|g|^2dxdx_{N+1}. 
\end{eqnarray}
	From (\ref{yu-7-31-4}), (\ref{yu-7-31-7}) and (\ref{yu-7-31-11}), we get 
\begin{eqnarray}\label{yu-7-31-12}
	C(r)\int_{B_{r_2}(x_0,0)\backslash B_{r_1}(x_0,0)}|g|^2dxdx_{N+1}
	\leq e^{2s(e^{144\lambda r^2}-e^{108\lambda r^2})}\int_{B_{r_1}(x_0,0)}|g|^2dxdx_{N+1}\nonumber\\
	+e^{2s(e^{95\lambda r^2}-e^{108\lambda r^2})}\int_{B_{r_3}(x_0,0)}|g|^2dxdx_{N+1}. 
\end{eqnarray}
	Fix $\lambda:=\lambda_0>0$ and denote  
\begin{equation*}\label{yu-7-31-13}
	\epsilon:=e^{2s(e^{95\lambda_0r^2}-e^{108\lambda_0 r^2})},\;\;
	\epsilon_0:=e^{2s_0(e^{95\lambda_0r^2}-e^{108\lambda_0 r^2})}
\end{equation*}
	and 
\begin{equation*}\label{yu-7-31-14}
	\mu:=\min_{r> 0}\frac{e^{144\lambda_0r^2}-e^{108\lambda_0 r^2}}{e^{108\lambda_0 r^2}-e^{95\lambda_0 r^2}}>0. 
\end{equation*}
	Note that, this minimum can be taken by the fact 
$$
	\lim_{r\to 0}\frac{e^{144\lambda_0r^2}-e^{108\lambda_0 r^2}}{e^{108\lambda_0 r^2}-e^{95\lambda_0 r^2}}=\frac{36}{13}.
$$
	Then, it follows from  (\ref{yu-7-31-12}) that 
\begin{multline}\label{yu-7-31-15}
	C(r)\int_{B_{r_2}(x_0,0)\backslash B_{r_1}(x_0,0)}|g|^2dxdx_{N+1}\\
	\leq \epsilon^{-\mu} \int_{B_{r_1}(x_0,0)}|g|^2dxdx_{N+1}
	+\epsilon\int_{B_{r_3}(x_0,0)}|g|^2dxdx_{N+1}.
\end{multline}

We treat in two cases separately.

$\bullet$ If 
\begin{equation*}\label{yu-7-31-16}
	\left(\frac{\int_{B_{r_1}(x_0,0)}|g|^2dxdx_{N+1}}{\int_{B_{r_3}(x_0,0)}|g|^2dxdx_{N+1}}\right)^{\frac{1}{1+\mu}}>\epsilon_0, 
\end{equation*}
	then 
\begin{eqnarray}\label{yu-7-31-17}
	&\;&C(r)\int_{B_{r_2}(x_0,0)\backslash B_{r_1}(x_0,0)}|g|^2dxdx_{N+1}\nonumber\\
	&\leq& C\left(\int_{B_{r_3}(x_0,0)}|g|^2dxdx_{N+1}\right)^{1-\frac{1}{1+\mu}}
	\left(\int_{B_{r_3}(x_0,0)}|g|^2dxdx_{N+1}\right)^{\frac{1}{1+\mu}}\nonumber\\
	&\leq& C\epsilon_0\left(\int_{B_{r_3}(x_0,0)}|g|^2dxdx_{N+1}\right)^{\frac{\mu}{1+\mu}}
	\left(\int_{B_{r_1}(x_0,0)}|g|^2dxdx_{N+1}\right)^{\frac{1}{1+\mu}}.
\end{eqnarray}

$\bullet$	If 
\begin{equation*}\label{yu-7-31-18}
	\left(\frac{\int_{B_{r_1}(x_0,0)}|g|^2dxdx_{N+1}}{\int_{B_{r_3}(x_0,0)}|g|^2dxdx_{N+1}}\right)^{\frac{1}{1+\mu}}\leq \epsilon_0, 
\end{equation*}
	we choose $s\geq s_0$ such that $\epsilon=\left(\frac{\int_{B_{r_1}(x_0,0)}|g|^2dxdx_{N+1}}{\int_{B_{r_3}(x_0,0)}|g|^2dxdx_{N+1}}\right)^{\frac{1}{1+\mu}}$. 
	Then, by (\ref{yu-7-31-15}), we have 
\begin{equation}\label{yu-7-31-19}
	C(r)\int_{B_{r_2}(x_0,0)\backslash B_{r_1}(x_0,0)}|g|^2dxdx_{N+1}
	\leq 2\left(\int_{B_{r_3}(x_0,0)}|g|^2dxdx_{N+1}\right)^{\frac{\mu}{1+\mu}}
	\left(\int_{B_{r_1}(x_0,0)}|g|^2dxdx_{N+1}\right)^{\frac{1}{1+\mu}}.
\end{equation}
	So, by (\ref{yu-7-31-17}) and (\ref{yu-7-31-19}), we get (\ref{yu-7-31-1-bb}) 
	with $\beta=\frac{1}{1+\mu}$. 
	The proof is completed. 
\end{proof}

\subsection{Proof of Proposition \ref{yu-theorem-7-5-1}}

\begin{proof}[\textbf{Proof of Proposition \ref{yu-theorem-7-5-1}}]
 Arbitrarily take $R\in(0,\min\{R_0,\rho\})$. Let $u_1$ and $u_2$ be accordingly the solution to 
\begin{equation*}\label{yu-7-4-4}
\begin{cases}
	l(x)\partial_tu_{1}-\mbox{div}(A(x)\nabla u_1)+b(x)u_1=0&\mbox{in}\;\;\triangle_R(x_0)\times(0,2T),\\
	u_1=u&\mbox{on}\;\;\partial \triangle_R(x_0)\times(0,2T),\\
	u_1(\cdot,0)=0 &\mbox{in}\;\;\triangle_{R}(x_0)
\end{cases}
\end{equation*}	
	and
\begin{equation*}\label{yu-7-4-5}
\begin{cases}
	l(x)\partial_tu_{2}-\mbox{div}(A(x)\nabla u_2)+b(x)u_2=0&\mbox{in}\;\;\triangle_R(x_0)\times(0,2T),\\
	u_2=0&\mbox{on}\;\;\partial \triangle_R(x_0)\times(0,2T),\\
	u_2(\cdot,0)=u(\cdot,0)&\mbox{in}\;\;\triangle_{R}(x_0).
\end{cases}
\end{equation*}
It is clear that
	$u=u_1+u_2$ in $\triangle_R(x_0)\times[0,2T]$.
By a standard energy estimate for  parabolic equations, we have
\begin{equation}\label{yu-7-4-7}
	\sup_{t\in[0,T]}\|u_2(\cdot,t)\|_{H^1(\triangle_R(x_0))}\leq Ce^{CT}\|u(\cdot,0)\|_{H^1(\triangle_R(x_0))}.
\end{equation}
Hence
\begin{equation}\label{yu-7-4-8}
	\sup_{t\in[0,T]}\|u_1(\cdot,t)\|_{H^1(\triangle_R(x_0))}\leq C(1+e^{CT})\sup_{t\in[0,T]}\|u(\cdot,t)\|_{H^1(\triangle_R(x_0))}.
\end{equation}

Fix arbitrarily $t_0\in(0,\frac{T}{2})$ and let $v_1$ be the solution of 
 \begin{equation*}\label{yu-11-29-4-jia}
\begin{cases}
    l(x)\partial_tv_1-\mbox{div}(A(x)\nabla v_1)+b(x)v_1=0&\mbox{in}\;\;\triangle_{R}(x_0)\times\mathbb{R}^+,\\
    v_1=\eta u_1&\mbox{on}\;\;\partial\triangle_R(x_0)\times\mathbb{R}^+,\\
    v_1(\cdot,0)=0&\mbox{in}\;\;\triangle_R(x_0),
\end{cases}
\end{equation*}
where $\eta$ is given by \eqref{yu-6-6-6}.
Clearly,
	 $u=v_1+u_2$ in $\triangle_R(x_0)\times[0,t_0]$. In particular, 
\begin{equation*}\label{yu-7-4-6}
	u(\cdot,t_0)=v_1(\cdot,t_0)+u_2(\cdot,t_0)\;\;\mbox{in}\;\;\triangle_R(x_0).
\end{equation*}
Define
\begin{equation*}\label{yu-6-18-5jia}
	\tilde{v}_1(\cdot,t):=
\begin{cases}
	v_1(\cdot,t)&\mbox{if}\;\;t\geq 0,\\
	0&\mbox{if}\;\;t<0,
\end{cases}
\end{equation*}
and
\begin{equation*}\label{yu-6-18-6jia}
	\hat{v}_1(x,\mu)=\int_{\mathbb{R}}e^{-i\mu t}\tilde{v}_1(x,t)dt\quad\text{for}\;\;(x,\mu)\in\triangle_R(x_0)\times\mathbb R.
\end{equation*}
Note 	from Lemma \ref{yu-lemma-6-10-1} that $\hat{v}_1$ is well defined.
\par
Let 
\begin{equation*}\label{yu-7-5-bb-1}
	\kappa:=\min\left\{\frac{1}{2},\frac{\sqrt{2}}{4e\Pi}\right\} \quad\text{with}\;\;\Pi\;\;\text{given in Lemma
	\ref{yu-lemma-6-18-1}}.
\end{equation*}
We  define
$$V=V_1+V_2\quad\mbox{in}\;\;\triangle_R(x_0)\times(-\kappa R,\kappa R),
$$
where 
\begin{equation}\label{yu-6-23-6jia}
	V_1(x,y)=\frac{1}{2\pi}\int_{\mathbb{R}}e^{it_0\mu}\hat{v}_1(x,\mu)
	\frac{\sinh(\sqrt{-i\mu}y)}{\sqrt{-i\mu}}d\mu\quad\mbox{in}\;\;\triangle_R(x_0)\times(-\kappa R,\kappa R)
\end{equation}
and 
\begin{equation}\label{yu-7-5-7}
	V_2(x,y)=\sum_{i=1}^\infty\alpha_ie^{-\mu_it_0}f_i(x)\frac{\sinh(\sqrt{\mu_i}y)}
	{\sqrt{\mu_i}}\;\;\mbox{in}\;\;\triangle_R(x_0)\times(-\kappa R,\kappa R)
\end{equation}
with 
\begin{equation*}\label{yu-7-5-6}
	\alpha_i=\int_{\triangle_R(x_0)}l(x)u_2(x,0)f_i(x)dx\;\;\mbox{for each}\;\;i\in\mathbb{N}^+,
\end{equation*}
and $\{\mu_i\}_{i=1}^\infty\subset\mathbb{R}^+$, $\{f_i\}_{i=1}^{\infty}\subset \mathcal{L}^2(\triangle_{R}(x_0))$  given by (\ref{yu-6-7-10}). Here we  note that from 
Lemma \ref{yu-lemma-6-18-1}, $V_1$ is also well defined. 
One can readily check that 
\begin{equation}\label{yu-7-5-9}
\begin{cases}
	\mbox{div}(A(x)\nabla V)+l(x)V_{yy}-b(x)V=0&\mbox{in}\;\;
	\triangle_{\frac{R}{2}}(x_0)\times(-\kappa R,\kappa R),\\
	V(x,0)=0&\mbox{in}\;\;\triangle_{\frac{R}{2}}(x_0),\\
	V_y(x,0)=u(x,t_0)&\mbox{in}\;\;\triangle_{\frac{R}{2}}(x_0).
\end{cases}
\end{equation}	
	By Lemma \ref{yu-lemma-7-4-2}, we have for any $r\in (0,\frac{1}{16} \kappa R)$, 
\begin{equation}\label{yu-7-4-10}
      \|V\|_{L^2(B_{6r}(x_0,0))}\leq C(r)\|V\|^\beta_{L^2(B_{r}(x_0,0))}\|V\|_{L^2(B_{8r}(x_0,0))}^{1-\beta}.
\end{equation}
Since $V_y$ also satisfies the first equation of (\ref{yu-7-5-9}),	by the interior estimate of elliptic equations
we find
\begin{eqnarray}\label{yu-7-4-11}
	\int_{B_{6r}(x_0,0)}|V|^2dxdy&\geq &Cr^2\int_{B_{11r/2}(x_0,0)}(|\nabla V|^2+|V_y|^2)dxdy\nonumber
	\\&\geq&\frac{Cr^2}{2}
	\left(\int_{B_{11r/2}(x_0,0)}|V_y|^2dxdy+\int_{B_{11r/2}(x_0,0)}|V_y|^2dxdy\right)\nonumber\\
	&\geq&Cr^2\left(\int_{B_{11r/2}(x_0,0)}|V_y|^2dxdy+r^2\int_{B_{5r}(x_0,0)}(|\nabla V_y|^2+|V_{yy}|^2)dxdy\right)\nonumber\\
	&\geq& Cr^3\left(\frac{1}{r}\int_{B_{5r}(x_0,0)}|V_y|^2dxdy+r\int_{B_{5r}(x_0,0)}|V_{yy}|^2dxdy\right).
\end{eqnarray}
     As a simple corollary of \cite[Lemma 9.9, page 315]{Brezis}, we have the following  trace theorem 
\begin{equation*}\label{yu-7-4-12}
	\int_{\triangle_{4r}(x_0)}|Q(x,0)|^2dx\leq C\left(\frac{1}{r}\int_{B_{5r}^+(x_0,0)}|Q|^2dxdy+r\int_{B^+_{5r}(x_0,0)}|Q_y|^2dxdy\right)
\end{equation*}
        for any $Q\in H^1(B_{5r}(x_0,0))$. Hence, by (\ref{yu-7-5-9}) and (\ref{yu-7-4-11}) we have 
\begin{eqnarray}\label{yu-7-4-13}
	Cr^3\int_{\triangle_{4r}(x_0)}|u(x,t_0)|^2dx\leq \int_{B_{6r}(x_0,0)}|V|^2dxdy.
\end{eqnarray}	
By Lemma \ref{yu-proposition-7-1-1},  we obtain that there is  $\gamma\in (0,1)$ such that for any $r\in(0,\frac{1}{3}\kappa R)$, 
\begin{equation}\label{yu-7-4-2}
	\|V\|_{L^2(B_r(x_0,0))}\leq Cr^{-2}\|V\|_{H^1(B_{2r}(x_0,0))}^{\gamma}\|u(\cdot,t_0)\|_{L^2(\triangle_{2r}(x_0))}^{1-\gamma}.
\end{equation}
Again, by the interior estimate, there is a constant $C>0$ such that 
\begin{equation}\label{yu-7-4-3}
	\|V\|_{H^1(B_{2r}(x_0,0))}\leq Cr^{-1} \|V\|_{L^2(B_{3r}(x_0,0))}.
\end{equation}
Hence, it follows from  (\ref{yu-7-4-2}) and (\ref{yu-7-4-3}) that
\begin{equation}\label{yu-7-4-1}
	\|V\|_{L^2(B_r(x_0,0))}\leq Cr^{-3}\|V\|_{L^2(B_{3r}(x_0,0))}^{\gamma}\|u(\cdot,t_0)\|_{L^2(\triangle_{2r}(x_0))}^{1-\gamma}.
\end{equation}
It follows from  (\ref{yu-7-4-13}), (\ref{yu-7-4-10}) and \eqref{yu-7-4-1} that  
\begin{equation}\label{yu-7-5-1}
	\|u(\cdot,t_0)\|_{L^2(\triangle_{4r(x_0)})}\leq C(r)r^{-3(\frac{1}{2}+\beta)}\|u(\cdot,t_0)\|_{L^2(\triangle_{2r}(x_0))}^{(1-\gamma)\beta}
	\|V\|^{1-(1-\gamma)\beta}_{L^2(B_{8r}(x_0,0))}.
\end{equation}
\par
\medskip

To finish the proof, it suffices to bound the term $\|V\|_{L^2(B_{8r}(x_0,0))}$. Recall that $V=V_1+V_2$, we will treat 
$V_1$ and $V_2$ separately.

In fact, we derive from (\ref{yu-6-23-6jia})  that
	for each $x\in \triangle_{8r}(x_0)\subset \triangle_{\frac{R}{2}}(x_0)$ and $|y|<\frac{\kappa R}{8}$, 
\begin{eqnarray*}\label{yu-7-5-2}
	|V_1(x,y)|
	&=&\left|\frac{1}{2\pi}\int_{\mathbb{R}}e^{it_0\mu}\hat{v}_1(x,\mu)\int_{-y}^ye^{\sqrt{-i\mu}s}dsd\mu\right|
	\leq \frac{1}{2\pi}\int_{\mathbb{R}}|\hat{v}_1(x,\mu)|\int_{-y}^y|e^{\sqrt{-i\mu}s}|dsd\mu
	\nonumber\\
	&\leq&\frac{\kappa R}{8\pi}\int_{\mathbb{R}}|\hat{v}_1(x,\mu)|e^{\frac{1}{8\sqrt{2}}\kappa\sqrt{|\mu|}R}d\mu\nonumber\\
	&\leq &\frac{\kappa R}{8\pi}\left(\int_{\mathbb{R}}|\hat{v}_1(x,\mu)|^2e^{\frac{3}{4\sqrt{2}}\kappa\sqrt{|\mu|}R}d\mu\right)^{\frac{1}{2}}
	\left(\int_{\mathbb{R}}e^{-\frac{1}{2\sqrt{2}}\kappa \sqrt{|\mu|}R}d\mu\right)^{\frac{1}{2}}\nonumber\\
	&=&\frac{\sqrt{2}}{4\pi}\left(\int_{\mathbb{R}}|\hat{v}_1(x,\mu)|^2e^{\frac{3}{4\sqrt{2}}\kappa\sqrt{|\mu|}R}d\mu\right)^{\frac{1}{2}}.
\end{eqnarray*}
	Hence, by Lemma \ref{yu-lemma-6-18-1} and (\ref{yu-7-4-8}), we have for each $r<\frac{R}{32}$,
\begin{eqnarray}\label{yu-7-5-3}
	\int_{\triangle_{8r}(x_0)}|V_1(x,y)|^2dx
	&\leq& CT^{-1}e^{CR^{1-N}(1+\frac{1}{T-t_0})T}G^2(R)
	\int_{\mathbb{R}}e^{-\frac{1}{4\sqrt{2}}\kappa \sqrt{|\mu|}R}d\mu\nonumber\\
	&\leq&\frac{CT^{-1}e^{CR^{1-N}(1+\frac{1}{T-t_0})(1+T)}G^2(R)}{R^2}.
\end{eqnarray}
	\par

While, by  (\ref{yu-7-5-7}) and (\ref{yu-7-4-7}) we obtain 
\begin{eqnarray}\label{yu-7-5-10}
	\int_{B_{8r}(x_0,0)}|V_2|^2dxdy&\leq&\Lambda_3 \int_{-8r}^{8r}\int_{\triangle_{R}(x_0)}l(x)|V_2|^2dxdy
	\leq \Lambda_3\int_{-8r}^{8r}\sum_{i=1}^\infty\alpha_i^2e^{-2\mu_it_0}\left|\frac{\sinh(\sqrt{\mu_i}y)}{\sqrt{\mu_i}}\right|^2dy\nonumber\\
	&\leq&2^{8}r^2\Lambda_3\left(1+e^{\frac{8\rho^2}{t_0}}\right)\sum_{i=1}^\infty\alpha_i^2\leq 
	Cr^2e^{\frac{C(1+T+T^2)}{t_0}}\int_{\triangle_R(x_0)}|u(x,0)|^2dx\nonumber\\
	&\leq& Ce^{\frac{C(1+T^2)}{t_0}}G^2(R).
\end{eqnarray}
	Therefore, by (\ref{yu-7-5-3}) and (\ref{yu-7-5-10}) we conclude that
\begin{equation*}\label{yu-7-5-11}
	\|V\|_{L^2(B_{8r}(x_0,0))}\leq CR^{-2}e^{\frac{CR^{1-N}(T^2+1)}{t_0}}G(R).
\end{equation*}
	This, together with (\ref{yu-7-5-1}), means that 
\begin{equation*}\label{yu-7-5-12}
	\|u(\cdot,t_0)\|_{L^2(\triangle_{4r})}
	\leq C(r)R^{-2[1-(1-\gamma)\beta]}e^{\frac{CR^{1-N}(T^2+1)}{t_0}[1-(1-\gamma)\beta]}
	\|u(\cdot,t_0)\|_{L^2(\triangle_{2r})}^{(1-\gamma)\beta}G^{1-(1-\gamma)\beta}(R).
\end{equation*}
	Taking $\sigma=(1-\gamma)\beta$ and using a scaling technique, the proof is immediately achieved.
	\end{proof}

\subsection{Proof of Proposition  \ref{yu-theorem-7-10-6}}

\begin{proof}[\textbf{Proof of Proposition \ref{yu-theorem-7-10-6}}]
We proceed the proof with three steps as follows.

\textbf{Step 1. In the interior.}   Let $K_1$ and $K_2$ be two compact subsets with non-empty interior   of $\Omega$.
Denoting $G_{\Omega}=\sup_{t\in[0,T]}\|u(\cdot,t)\|_{H^1(\Omega)}$, we shall show that 
\begin{equation}\label{du919}
\|u(\cdot,t_0)\|_{L^2(K_1)}\leq e^{\frac{C(T^2+1)}{t_0}}\|u(\cdot,t_0)\|^{\sigma}_{L^2(K_2)}G_{\Omega}^{1-\sigma}.
\end{equation}
In fact, there exists a sequence of balls $\{\triangle_{r}(x_i)\}_{j=0}^{p}$ such that 
\begin{equation*}\label{covering}
K_1\subset\bigcup_{i=1}^{p}\triangle_{r}(x_i)\subset\Omega, \;\;\;\;
\triangle_{r}(x_0)\subset K_2,
\end{equation*}
and for each $1\leq i\leq p$, there exists a chain
of balls $\triangle_{r}(x_i^j)$, $1\leq j\leq n_i$, such that
\begin{equation*}\label{chain}
\begin{split}
&\triangle_r(x_i^{1})=\triangle_{r}(x_i),\;\;\triangle_r(x_i^{n_i})=\triangle_{r}(x_0),\\
&\triangle_{r}(x_i^j)\subset \triangle_{2r}(x_i^{j+1})\subset\Omega,\;\;1\leq j\leq
n_i-1.
\end{split}
\end{equation*}
By Proposition~\ref{yu-theorem-7-5-1}, we obtain that there are constants 
$N_i^j= N_i^j(r,p)\geq 1$ and $\theta_i^j=\theta_i^j(r,p)\in(0,1)$ such
that
\begin{equation*}
\|u(\cdot,t_0)\|_{L^2(\triangle_r(x_i^j))}\leq \|u(\cdot,t_0)\|_{L^2(\triangle_{2r}(x_i^{j+1}))}\leq e^{\frac{N_i^j(T^2+1)}{t_0}}\|u(\cdot, t_0)\|_{L^2(\triangle_{r}(x_i^{j+1}))}^{\theta_i^j}
G_\Omega^{1-\theta_i^j}.
\end{equation*}
Iterating the above procedure, we derive that there are constants 
$N_i=N_i(K_1,K_2,p)\geq 1$ and $\theta_i=\theta_i(K_1,K_2,p)\in(0,1)$
such that
\begin{equation*}
\|u(\cdot,t_0)\|_{L^2(\triangle_{r}(x_i))}\leq e^{\frac{N_i(T^2+1)}{t_0}}\|u(\cdot, t_0)\|_{L^2(\triangle_{r}(x_0))}^{\theta_i}
G_\Omega^{1-\theta_i}.
\end{equation*}
Hence, \eqref{du919} follows.

\medskip

\textbf{Step 2. Flattening the boundary and taking the even reflection.}
Arbitrarily fix $x_0\in\partial\Omega$.  Without loss of generality, we may assume that $A(x_0)=I$.
Following the arguments to flatten locally the boundary as in \cite{Adolfsson-Escauriaza} (see also 
\cite{Canuto-Rosset-Vessella}),  we have that 
there exists a $C^1$-diffeomorphism $\Phi$ from $\triangle_{r_2}(0)$ to $\triangle_{r_1}(x_0)$ such that
\begin{equation*}\label{yu-7-27-1}
	\Phi(y', 0)\in \partial\Omega\cap \triangle_{r_1}(x_0)\;\;\mbox{for each}\;\;y'\in\triangle'_{r_2}(0),
\end{equation*}
\begin{equation*}\label{yu-7-27-4}
	\Phi(\triangle_{r_2}^{+}(0))\subset \triangle_{r_1}(x_0)\cap\Omega,
\end{equation*} 
\begin{equation}\label{yu-10-10-1}
C^{-1}\leq\text{det} J\Phi(y)\leq C\quad\text{for each}\;\;y\in\triangle_{r_2}(0),
\end{equation}
\begin{equation}\label{yu-10-10-2}
|\text{det} J\Phi(y)-\text{det} J\Phi(\tilde y)|\leq C|y-\tilde y|\quad\text{for each}\;\;y,\tilde y\in\triangle_{r_2}(0),
\end{equation}
\begin{equation*}\label{yu-10-10-3}
C^{-1}|y-\tilde y|\leq |\Phi(y)-\Phi(\tilde y)|\leq C|y-\tilde y|\quad\text{for each}\;\;y,\tilde y\in\triangle_{r_2}(0),
\end{equation*}
\begin{equation}\label{du8101}
 \tilde{a}_{jN}(y',0)=\tilde a_{Nj}(y',0)=0\;\;\;\text{for each}\;\;y'\in \triangle_{r_2}'(0), \;\;j=1,\dots, N-1,
\end{equation}
	where 
\begin{equation*}\label{yu-7-8-1}
	\tilde{A}(y)=[\tilde a_{ij}]_{N\times N}=\mbox{det}J\Phi(y)(J\Phi^{-1})(\Phi(y))^{tr}A(\Phi(y))(J\Phi^{-1})(\Phi(y)),\quad
y\in \triangle_{r_2}^{+}(0).
\end{equation*} 
\par
       By (\ref{yu-10-10-1}) and (\ref{yu-10-10-2}), one can check that $\tilde A(\cdot)$ satisfies the uniform ellipticity condition and the Lipschitz condition in $\triangle_{r_2}^{+}(0)$. 
         Denoting 
\begin{equation*}\label{yu-7-8-2}
z(y,t)=u(\Phi(y),t),\;\;\tilde{b}(y)=\mbox{det}J\Phi(y)b(\Phi(y))\quad\text{for each}\;\;y\in \triangle_{r_2}^{+}(0),\;\;t\in(0,2T),
\end{equation*}
by (\ref{yu-10-10-1}) we have
\begin{equation*}\label{yu-7-8-5}
	\tilde{b}(\cdot)\;\;\mbox{satisfies (\ref{yu-6-24-1-1-b}) in}\;\;\triangle_{r_2}^+(0),
\end{equation*}
\begin{equation*}\label{du8102}
\left\{
\begin{split}
\mbox{det}J\Phi(y)z_t(y,t)-\text{div}(\tilde A(y)\nabla z(y,t))+\tilde b(y)=0\;\;\;\;\text{in}\;\;  \triangle^+_{r_2}(0)	\times(0,2T),  \\
\frac{\partial z}{\partial y_N}=0\;\;\;\;\text{on}\;\;  (\triangle'_{r_2}(0)\times\{0\})	\times(0,2T).  \\
\end{split}\right.
\end{equation*}
For any $y=(y',y_N)\in\triangle_{r_2}(0)$, using the even reflection and denoting $\check{A}(y)=[\check{a}^{ij}(y)]_{N\times N}$ by 
\begin{equation*}\label{yu-7-8-6}
\begin{cases}
	\check{a}_{ij}(y',y_N)=\tilde{a}_{ij}(y',|y_N|),&\mbox{if}\;\;1\leq i,j\leq N-1,\;\;\mbox{or}\;\;i=j=N,\\
	\check{a}_{Nj}(y',y_N)=\check{a}_{jN}(y',y_N)=\mbox{sign}(y_N)\tilde{a}_{jN}(y',|y_N|),
	&\mbox{if}\;\;1\leq j\leq N-1,
\end{cases}
\end{equation*}
\begin{equation*}\label{yu-7-8-8}
	\check{b}(y',y_N)=\tilde{b}(y',|y_N|),\;\;	\check{l}(y',y_N)=\mbox{det}J\Phi(y',|y_N|),
\end{equation*}
and
\begin{equation*}\label{yu-7-8-9}
	Z(y,t)=z(y',|y_N|,t)\;\;\mbox{for each}\;\;(y,t)\in\triangle_{r_2}(0)\times(0,2T),
\end{equation*}
      By (\ref{du8101}), we see that $\check{A}$ verifies the uniform ellipticity condition and the Lipschitz condition in $\triangle_{r_2}(0)$,  $\check{b}(\cdot)$ verifies (\ref{yu-6-24-1-1-b}) in $\triangle_{r_2}(0)$, 
\begin{equation*}\label{yu-7-8-8}
C^{-1}\leq l(y)\leq C,\;\;		|l(y)-l(\tilde y)|\leq C|y-\tilde y|\quad\mbox{for each}\;\;y,\tilde y\in\triangle_{r_2}(0),
\end{equation*}
	and that 
\begin{equation}\label{yu-7-8-10}
	\check{l}(y)Z_t(y,t)-\mbox{div}(\check{A}(y)\nabla Z(y,t))+\check{b}(y)=0\;\;\mbox{in}\;\;\triangle_{r_2}(0)\times(0,2T).
\end{equation}
	
Let $\hat y=(0',r_2/2)$.   
For each $0<r\leq r_2/8$,  by applying Proposition \ref{yu-theorem-7-5-1} to the solution $Z$ of (\ref{yu-7-8-10}), 
 similar to the proof of Step 1, we obtain
\begin{equation*}\label{yu-7-8-b-1}
	\|Z(\cdot,t_0)\|_{L^2(\triangle_{2r}(0))}
	\leq C(r)e^{\frac{C(T^2+1)}{t_0}}
	\|Z(\cdot,t_0)\|_{L^2(\triangle_{r_2/4}(\hat y))}^{\sigma_1} G_1^{1-\sigma_1}\left(\triangle_{r_2}(0)\right),
\end{equation*}
	where  $G_1(\triangle_{r_2}(0))=\sup_{s\in[0,T]}\|Z(\cdot,s)\|_{H^1(\triangle_{r_2}(0))}$.	
Hence,
\begin{equation*}\label{yu-7-8-b-1}
	\|z(\cdot,t_0)\|_{L^2(\triangle^+_{2r}(0))}
	\leq C(r)e^{\frac{C(T^2+1)}{t_0}}
	\|z(\cdot,t_0)\|_{L^2(\triangle_{r_2/4}(\hat y))}^{\sigma_1} G_2^{1-\sigma_1}\left(\triangle_{r_2}(0)\right)
\end{equation*}	
        where $G_2(\triangle_{r_2}(0))=\sup_{s\in[0,T]}\|z(\cdot,s)\|_{H^1(\triangle_{r_2}^+(0))}$. Since the map $\Phi$ is $C^{1}$-diffeomorphism, we obtain that there exist $r_3>0$ and $\rho>0$ such that 
\begin{equation*}\label{yu-7-8-b-1}
	\|u(\cdot,t_0)\|_{L^2(\Omega\cap\triangle_{r_3}(x_0))}
	\leq C(r)e^{\frac{C(T^2+1)}{t_0}}
	\|u(\cdot,t_0)\|_{L^2(\Omega_{\rho})}^{\sigma_1} G_\Omega^{1-\sigma_1}.
\end{equation*}

\medskip

\textbf{Step 3. Completing the proof.}
When $\Gamma$ is a neighborhood of
$\partial\Omega$ in $\Omega$,  there are a sequence  $\{x_j\}_{j=1}^{p}\subset\partial\Omega$ and a sequence $\{\triangle_{r_j}(x_j)\}_{j=1}^{p}$ such that 
$$\Gamma\subset \bigcup_{j=1}^{p}(\Omega\cap \triangle_{r_j}(x_j)).$$
By the result in Step 2 and a finite covering argument,  we first have 
\begin{equation}\label{dujiu1}
\|u(\cdot,t_0)\|_{L^2(\Gamma)}\leq  Ce^{\frac{C(T^2+1)}{t_0}}
	\|u(\cdot,t_0)\|_{L^2(\Omega_{\rho})}^{\sigma_1} G_\Omega^{1-\sigma_1}
\end{equation}
with some $\rho>0$.
By the result in Step 1, we then have
\begin{equation*}
\|u(\cdot,t_0)\|_{L^2(\Omega_\rho)}\leq  Ce^{\frac{C(T^2+1)}{t_0}}
	\|u(\cdot,t_0)\|_{L^2(\omega)}^{\sigma_1} G_\Omega^{1-\sigma_1}.
\end{equation*}
This, together with \eqref{dujiu1}, indicates that 
\begin{equation*}
\|u(\cdot,t_0)\|_{L^2(\Gamma)}\leq  Ce^{\frac{C(T^2+1)}{t_0}}
	\|u(\cdot,t_0)\|_{L^2(\omega)}^{\sigma_2} G_\Omega^{1-\sigma_2}.
\end{equation*}
Which, combined with the result in Step 1 again, implies the desired estimate and completes the proof.
\end{proof}

\section{Appendix}
	
\subsection{Proof of Lemma \ref{yu-lemma-7-1-1}}
By the definition of $W$, we have  
\begin{eqnarray*}\label{yu-7-1-8}
	\nabla V&=&\nabla(\theta^{-1}W)=W\nabla \theta^{-1}+\theta^{-1}\nabla W\nonumber\\
	&=&-s\lambda \alpha\theta^{-1}W\nabla \varphi+\theta^{-1}\nabla W.
\end{eqnarray*}
	Therefore, 
\begin{eqnarray}\label{yu-7-1-9}
	-\theta\mbox{div}(\bar{A}\nabla V)&=& s\lambda^2\alpha W\nabla\varphi\cdot(\bar{A}\nabla \varphi)
	-s^2\lambda^2\alpha^2W\nabla\varphi\cdot(\bar{A}\nabla \varphi)\nonumber\\
	&\;&+2s\lambda\alpha\nabla W\cdot(\bar{A}\nabla \varphi)+s\lambda\alpha  W\mbox{div}(\bar{A}\nabla \varphi)
	-\mbox{div}(\bar{A}\nabla W).
\end{eqnarray}
	Let 
\begin{equation*}\label{yu-7-1-10}
\begin{cases}
	I_1:=-\mbox{div}(\bar{A}\nabla W)-s^2\lambda^2\alpha^2W\nabla\varphi\cdot(\bar{A}\nabla\varphi),\\
	I_2:=2s\lambda\alpha\nabla W\cdot(\bar{A}\nabla\varphi)+2s\lambda^2\alpha W\nabla\varphi\cdot(\bar{A}\nabla\varphi),\\
	I_3:=-\theta\mbox{div}(\bar{A}\nabla V)-s\lambda\alpha W\mbox{div}(\bar{A}\nabla \varphi)+s\lambda^2\alpha W\nabla \varphi\cdot(\bar{A}\nabla\varphi).
\end{cases}
\end{equation*}
	By (\ref{yu-7-1-9}), it is clear that $I_1+I_2=I_3$. Then 
\begin{equation}\label{yu-7-1-11}
	I_1I_2\leq \frac{1}{2}|I_3|^2.
\end{equation}
	For the term $|I_3|^2$, we have 
\begin{eqnarray}\label{yu-7-1-14}
	\frac{1}{2}|I_3|^2&\leq& \theta^2|\mbox{div}(\bar{A}\nabla V)|^2+2s^2\lambda^2\alpha^2|W|^2|\mbox{div} (\bar{A}\nabla\varphi)|^2\nonumber\\
	&\;&+2s^2\lambda^4\alpha^2|\nabla\varphi\cdot(\bar{A}\nabla\varphi)|^2|W|^2.
\end{eqnarray}
	By (\ref{yu-7-1-10}), we have 
\begin{eqnarray}\label{yu-7-1-15}
	I_1I_2&=&-2s\lambda\alpha[\mbox{div}(\bar{A}\nabla W)+s^2\lambda^2\alpha^2W\nabla\varphi\cdot(\bar{A}\nabla \varphi)][\nabla W\cdot(\bar{A}\nabla \varphi)+\lambda W\nabla\varphi\cdot (\bar{A}\nabla \varphi)]\nonumber\\
	&=&-2s\lambda^2\alpha W[\mbox{div}(\bar{A}\nabla W)+s^2\lambda^2\alpha^2W\nabla\varphi\cdot(\bar{A}\nabla\varphi)]\nabla\varphi\cdot(\bar{A}\nabla\varphi)\nonumber\\
	&\;&-2s\lambda\alpha\mbox{div}(\bar{A}\nabla W)\nabla W\cdot(\bar{A}\nabla \varphi)-2s^3\lambda^3\alpha^3W\nabla W\cdot (\bar{A}\nabla\varphi)\nabla\varphi\cdot(\bar{A}\nabla\varphi)\nonumber\\
	&:=&\sum_{i=1}^3J_i.
\end{eqnarray}	
	Next, we compute the terms $J_i$ one by one. For the term $J_1$, we have 
\begin{eqnarray}\label{yu-7-1-16}
	J_1&=&-2s^3\lambda^4\alpha^3|W|^2|\nabla\varphi\cdot(\bar{A}\nabla\varphi)|^2
	+2s\lambda^2\alpha\nabla W\cdot(\bar{A}\nabla W)\nabla\varphi\cdot(\bar{A}\nabla\varphi)\nonumber\\
	&\;&+2s\lambda^2 W\nabla[\alpha\nabla\varphi\cdot(\bar{A}\nabla\varphi)]\cdot(\bar{A}\nabla W)
	-2s\lambda^2\mbox{div}[\alpha W\bar{A}\nabla W\nabla\varphi\cdot(\bar{A}\nabla\varphi)].
\end{eqnarray}
	Moreover, 
\begin{eqnarray*}\label{yu-7-1-17}
	J_2&=&2s\lambda\alpha\nabla[\nabla W\cdot(\bar{A}\nabla\varphi)]\cdot(\bar{A}\nabla W)
	+2s\lambda^2\alpha|\nabla W\cdot(\bar{A}\nabla\varphi)|^2\nonumber\\
	&\;&-2s\lambda\mbox{div}[\alpha \bar{A}\nabla W\nabla W\cdot(\bar{A}\nabla\varphi)].
\end{eqnarray*}
	Note that 
\begin{eqnarray*}\label{yu-7-1-19}
	&\;&\nabla[\nabla W\cdot(\bar{A}\nabla\varphi)]\cdot(\bar{A}\nabla W)\nonumber\\
	&=&(\bar{A}\nabla W)\cdot[D^2W(\bar{A}\nabla\varphi)]
	+(\bar{A}\nabla W)[D^2\varphi (\bar{A}\nabla W)]\nonumber\\
	&\;&+\left(\sum_{i,j=1}^{N+1}\partial_iW\nabla \bar{a}^{ij}\partial_j\varphi\right)\cdot(\bar{A}\nabla W)\nonumber\\
	&=&\frac{1}{2}\nabla[\nabla W\cdot(\bar{A}\nabla W)]\cdot(\bar{A}\nabla\varphi)+(\bar{A}\nabla W)[D^2\varphi (\bar{A}\nabla W)]\nonumber\\
	&\;&+\left(\sum_{i,j=1}^{N+1}\partial_iW\nabla \bar{a}^{ij}\partial_j\varphi\right)\cdot(\bar{A}\nabla W)
	-\frac{1}{2}\left(\sum_{i,j=1}^{N+1}\partial_iW\nabla \bar{a}^{ij}\partial_jW\right)\cdot(\bar{A}\nabla\varphi).
\end{eqnarray*}
	Hence
\begin{eqnarray}\label{yu-7-1-20}
	J_2&=&2s\lambda^2\alpha|\nabla W\cdot(\bar{A}\nabla\varphi)|^2
	-s\lambda^2\alpha\nabla\varphi\cdot(\bar{A}\nabla\varphi)\nabla W\cdot(\bar{A}\nabla W)\cdot\nonumber\\
	&\;&-s\lambda\alpha\mbox{div}(\bar{A}\nabla\varphi)\nabla W\cdot (\bar{A}\nabla W)+2s\lambda\alpha(\bar{A}\nabla W)\cdot[D^2\varphi(\bar{A}\nabla W)]\nonumber\\
	&\;&{+2s\lambda\alpha\left(\sum_{i,j=1}^{N+1}\partial_iW\nabla \bar{a}^{ij}\partial_j\varphi\right)\cdot(\bar{A}\nabla W)-s\lambda\alpha\left(\sum_{i,j=1}^{N+1}\partial_iW\nabla \bar{a}^{ij}\partial_jW\right)\cdot(\bar{A}\nabla\varphi)}
	\nonumber\\
	&\;&-2s\lambda\mbox{div}[\alpha \bar{A}\nabla W\nabla W\cdot(\bar{A}\nabla\varphi)]+s\lambda\mbox{div}[\alpha\nabla W\cdot(\bar{A}\nabla W)\bar{A}\nabla\varphi].
\end{eqnarray}
         Further, 
\begin{eqnarray}\label{yu-7-1-18}
	J_3&=&3s^3\lambda^4\alpha^2|W|^2|\nabla\varphi\cdot(\bar{A}\nabla\varphi)|^2
	+s^3\lambda^3\alpha^3|W|^2\mbox{div}\{\bar{A}\nabla\varphi[\nabla\varphi\cdot (\bar{A}\nabla\varphi)]\}\nonumber\\
	&\;&-s^3\lambda^3\mbox{div}[\alpha^3|W|^2\bar{A}\nabla\varphi\nabla\varphi\cdot (\bar{A}\nabla\varphi)].
\end{eqnarray}
	Finally, by (\ref{yu-7-1-11})--(\ref{yu-7-1-18}) we obtain the desired identity and  complete  the proof.

\subsection{Proofs of some useful inequalities}
\subsubsection{Proof of (\ref{yu-9-29-1})}
	For each $h\in L^{\frac{N}{2}+\eta}(\triangle_r(x_0))$ and $f\in H^1(\triangle_r(x_0))$, we let $\theta=\frac{1}{r}(x-x_0)\in\triangle_1(0)$, 
	$\tilde{h}(\theta)=h(r\theta+x_0)=h(x)$ and $\tilde{f}(\theta)=f(x)$ similarly.
	One can check that 
$$
	\int_{\triangle_1(0)}|\tilde{h}||\tilde{f}|^2d\theta =r^{-N}\int_{\triangle_r(x_0)} |h||f|^2dx,
$$
$$
	\|\tilde{h}\|_{L^{\frac{N}{2}+\eta}(\triangle_1(0))}=r^{-\frac{2N}{N+2\eta}}
	\|h\|_{L^{\frac{N}{2}+\eta}(\triangle_r(x_0))},
$$
	and 
$$
	\|\tilde{f}\|_{L^2(\triangle_1(0))}=r^{-\frac{N}{2}}\|f\|_{L^2(\triangle_r(x_0))}.
$$
	Moreover, when $r\in(0,1)$, 
\begin{eqnarray*}
	\|f\|_{H^1(\triangle_r(x_0))}^2&=&r^N\int_{\triangle_1(0)}\left(\frac{1}{r^2}|\nabla_\theta\tilde{f}|^2+|\tilde{f}|^2\right)d\theta\nonumber\\
	&\geq& r^N\int_{\triangle_1(0)}(|\nabla_\theta \tilde{f}|^2+|\tilde{f}|^2)d\theta
	=r^N\|\tilde{f}\|^2_{H^1(\triangle_1(0))},
\end{eqnarray*}
	i.e., 
$$
	\|\tilde{f}\|_{H^1(\triangle_1(0))}\leq r^{-\frac{N}{2}}\|f\|_{H^1(\triangle_r(x_0))}. 
$$
	Therefore, by (\ref{yu-9-26-2}), we have 
\begin{eqnarray*}
        &\;&r^{-N}\int_{\triangle_r(x_0)}|h||f|^2dx=\int_{\triangle_1(0)}|\tilde{h}||\tilde{f}|^2d\theta\nonumber\\
        &\leq& \Gamma_2(\triangle_1(0),N,\eta)\|\tilde{h}\|_{L^{\frac{N}{2}+\eta}(\triangle_1(0))}
        \|\tilde{f}\|_{L^2(\triangle_1(0))}^{\frac{4\eta}{N+2\eta}}\|\tilde{f}\|^{\frac{2N}{N+2\eta}}_{H^1(\triangle_1(0))}\nonumber\\
        &\leq&\Gamma_2(\triangle_1(0), N,\eta)r^{-\frac{2N}{N+2\eta}}r^{-N}
       \|h\|_{L^{\frac{N}{2}+\eta}(\triangle_r(x_0))} \|f\|_{L^2(\triangle_r(x_0))}^{\frac{4\eta}{N+2\eta}}\|f\|^{\frac{2N}{N+2\eta}}_{H^1(\triangle_r(x_0))}. 
\end{eqnarray*}
	This implies that 
$$
	\int_{\triangle_r(x_0)}|h||f|^2dx\leq \Gamma_2(\triangle_1(0), N,\eta)r^{-\frac{2N}{N+2\eta}}
	 \|h\|_{L^{\frac{N}{2}+\eta}(\triangle_r(x_0))} \|f\|_{L^2(\triangle_r(x_0))}^{\frac{4\eta}{N+2\eta}}\|f\|^{\frac{2N}{N+2\eta}}_{H^1(\triangle_r(x_0))}.
$$
	Then (\ref{yu-9-29-1}) holds.
\subsubsection{Proof of (\ref{yu-6-22-8})}
	Indeed, for any $f\in H^1(I)$,  taking any $x\in I$, we have 
\begin{eqnarray*}
	|f(x)|^2\leq 2\left|\int_y^xf'(s)ds\right|^2+2|f(y)|^2
	\leq 2|I|\int_I|f'(s)|^2ds+2|f(y)|^2
\end{eqnarray*}
	for each $y\in I$. Integrating it with respect to $y$ over $I$, we get 
$$
	|f(x)|^2\leq 2|I|\int_{I}|f'(s)|^2ds+\frac{2}{|I|}\int_{I}|f(s)|^2ds. 
$$
	This implies (\ref{yu-6-22-8}).

\end{document}